\documentclass{amsart}

\usepackage{graphicx}
\usepackage{hyperref}
\usepackage{subfig}
\usepackage{amsmath}
\usepackage{mathrsfs} 
\usepackage{amsfonts}
\usepackage{amssymb}%
\usepackage{amsthm}
\usepackage{amscd}
\usepackage{verbatim}
\usepackage{url}
\usepackage{mathabx}
\usepackage{dsfont}
\usepackage[all]{xypic} 
\usepackage{endnotes}
\usepackage{enumerate}
\usepackage{mathtools}
\usepackage{comment}
\usepackage{tikz}

\newcommand{\mcM}{\mathcal{M}}
\newcommand{\mcC}{\mathcal{C}}
\newcommand{\mcL}{\mathcal{L}}
\newcommand{\mcD}{\mathcal{D}}

\newcommand{\mcS}{\mathcal{S}}

\newcommand{\scrT}{\mathscr{T}}
\newcommand{\scrM}{\mathscr{M}}

\newcommand{\C}{\mathbb{C}}

\newcommand{\N}{\mathbb{N}}

\newcommand{\ZZ}{\mathbb{Z}}

\newcommand{\isom}{\cong}

\newcommand{\ot}{\otimes}
\newcommand{\by}{\! \times \!}
\newcommand{\ra}{\rightarrow}
\newcommand{\rae}{\!\rightarrow\!}

\newcommand{\tr}{\operatorname{tr}}

\newcommand{\sgn}{\operatorname{sgn}}

\newcommand{\id}{\operatorname{id}}

\newcommand{\Vect}{{\rm Vect}}
\newcommand{\Rel}{{\rm Rel}}
\newcommand{\Set}{{\rm Set}}
\newcommand{\fdVect}{\Vect}

\newcommand{\sPf}{\operatorname{sPf}}
\newcommand{\sPfd}{\operatorname{sPf}^{\vee}}
\newcommand{\Pf}{\operatorname{Pf}}
\newcommand{\sDet}{\operatorname{sDet}}

\newcommand{\Mat}{\operatorname{Mat}}

\newcommand{\unit}{\mathds{1}}
\newcommand{\Ob}{\text{Ob}}
\newcommand{\Mor}{\text{Mor}}
\newcommand{\Hom}{\text{Hom}}

\newcommand{\msP}{\mathscr{P}}

\newcommand{\tn}{\textnormal}

\newcommand{\Tr}{\textnormal{Tr}}

\newcommand{\bra}[1]{\mbox{$\langle #1|$}}
\newcommand{\ket}[1]{\mbox{$|#1\rangle$}}

\newcommand{\ketbra}[2]{\mbox{$|#1\rangle\langle #2|$}}

\newcommand{\indexset}[1]{\mathsf{#1}}

\theoremstyle{plain}
\newtheorem{theorem}{Theorem}[section]
\newtheorem{corollary}[theorem]{Corollary}
\newtheorem{proposition}[theorem]{Proposition}
\newtheorem{lemma}[theorem]{Lemma}
\newtheorem{remark}[theorem]{Remark}
\newtheorem{observation}[theorem]{Observation}

\newtheorem{problem}[theorem]{Problem}
\theoremstyle{definition}
\newtheorem{defn}[theorem]{Definition}
\theoremstyle{definition}

\theoremstyle{definition}
\newtheorem{example}[theorem]{Example}
\theoremstyle{definition}
\newtheorem{algorithm}{Algorithm}

\specialcomment{smallblue}{\begingroup\color{blue}\footnotesize}{\endgroup}

\specialcomment{marginnote}{\begin{marginpar}}{\end{marginpar}}
\excludecomment{marginnote}

\title[Determinantal Circuits]{Generalized counting constraint satisfaction problems with determinantal circuits}
\author[Morton-Turner]{Jason Morton$^1$ and Jacob Turner$^1$}
\address{$^1$Department of Mathematics, Pennsylvania State University, University Park PA 16802}
\date{June 24, 2014}
\begin{document}

\maketitle
\begin{abstract}
Generalized counting constraint satisfaction problems include Holant problems with planarity restrictions; polynomial-time algorithms for such problems include matchgates and matchcircuits, which are based on Pfaffians.  
In particular, they use gates which are expressible in terms of a vector of sub-Pfaffians of a skew-symmetric matrix. 
We introduce a new type of circuit based instead on determinants, with seemingly different expressive power.  In these {\em determinantal circuits}, a gate is represented by the vector of all minors of an arbitrary matrix.  Determinantal circuits permit a different class of gates.  Applications of these circuits include proofs of theorems from algebraic graph theory including the Chung-Langlands formula for the number of rooted spanning forests of a graph and computing Tutte Polynomials of certain matroids. They also give a strategy for simulating quantum circuits with closed timelike curves.  Monoidal category theory provides a useful language for discussing such counting problems, turning combinatorial restrictions into categorical properties. We  introduce the counting problem in monoidal categories and count-preserving functors as a way to study $\mathsf{FP}$ subclasses of problems in settings which are generally $\mathsf{\#P}$-hard. Using this machinery we show that, surprisingly, determinantal circuits 
can be simulated by Pfaffian circuits at quadratic cost.  
\end{abstract}


\section {Introduction}

Let $\fdVect_\C$ be the category of finite-dimensional vector spaces and linear transformations over the base field $\mathbb{C}$.  A string diagram ~\cite{joyal1991geometry} in  $\fdVect_\C$ is a {\em tensor (contraction) network}.  Fixing such a diagram, the problem of computing the morphism represented is the {\em tensor contraction problem}, which is in general $\mathsf{\#P}$-hard (examples include weighted counting constraint satisfaction problems ~\cite{cai2012complexity}).  

We study complex-valued tensor contraction problems in subcategories of $\fdVect_\C$ by considering them as diagrams in a monoidal category.  For a survey of the rich diagrammatic languages that can be specified similarly see ~\cite{selinger2009survey} and the references therein.  
By a {\em circuit} we mean a combinatorial counting problem expressed as a string diagram in a monoidal subcategory of  $\fdVect_\C$ (that is, a tensor contraction network).  Such diagrams generalize weighted constraint satisfaction problems and Boolean circuits (such as by requiring planarity), and are often  related to existing description languages.  Subcategories of  $\fdVect_\C$ can faithfully represent Boolean~\cite{lafont2003towards} and quantum circuits~\cite{bergholm2011categorical}, counting constraint satisfaction problems, and many other problems~\cite{Damm03}.  

Suppose we have a problem $\mcL$; a common example are counting constraint satisfaction problems ~\cite{bulatov2010complexity}, perhaps with some restrictions such as planarity.  Such a problem can be described by the data of a monoidal word (see e.g.\ ~\cite{kassel1995quantum}, Chapter 12) and a {\em interpretation}
~\cite{selinger2009survey} map $i\!:\!\mcL \rae \mcC$ that assigns values to primitive terms in the word.  Then determining which morphism is obtained is a tensor contraction problem in some monoidal category $\mcC$. 

From the point of view of complexity theory, we are interested in the class $\mathsf{FP}$ which is comprised of the functions $\{0,1\}^* \ra \N$ computable by a deterministic polynomial-time Turing machine (see e.g.\ \cite[p. 344]{arora2009computational}).
A second functor $\mathbf{h}\!:\!\mcC \ra \mcS$ from a category $\mcC$ in which the contraction problem (Problem \ref{prob:count}) is in $\mathsf{FP}$ and a subcategory $\mcS$ of $\fdVect_\C$ that preserves the solution to the $\mathsf{FP}$ problem serves to characterize the problems which can be solved in polynomial time according to a particular contraction scheme.  

The motivation of this paper comes from holographic algorithms ~\cite{QCtcbSiPT} and our attempts to generalize it and give it a uniform language. This and related schemes work by exploiting some combinatorial identity or {\em kernel} relating an exponential sum (corresponding to performing the tensor 
contraction by a na\"ive algorithm) and a polynomial time operation that yields the same result. They can be viewed as a complementary alternative method to geometric complexity theory ~\cite{mulmuley2001geometric} in the study of which counting problems (such as computing a permanent) may be embedded in a determinant computation at polynomial cost.

We formulate a class of circuits based on determinants and show that the corresponding tensor contraction problem is solvable in polynomial time. The existence of such a class was conjectured in ~\cite{landsberg2012holographic}. A circuit class based on Pfaffians of minors had already been given ~\cite{morton2010pfaffian} and the formula for the number of rooted spanning forests of a graph ~\cite{MR1401006} hinted at the kernel to use for determinantal circuits. Indeed, we can recover the theorem using determinantal circuits.

Another motivation for the construction of determinantal circuits is that generalizations of holographic algorithms are of interest for their potential to find new or improved polynomial time algorithms for $\#\mathsf{P}$ problems. We also discuss applications of determinantal circuits to improve algorithms for computing certain Tutte polynomials that cannot be achieved with Pfaffian circuits.

We then explore the relationship between Pfaffian circuits ~\cite{morton2010pfaffian} (and so matchgates) and determinantal circuits. Since we found the construction of determinantal circuits to be most natural using the language of categories, we first have to recast Pfaffian circuits in the same language. We prove a functorial relationship between them.  We show that, surprisingly, every determinantal circuit can be expressed as a Pfaffian circuit at quadratic cost.

This paper is organized as follows.  In Section \ref{sec:CCC}, we describe the counting problem in monoidal categories and the setting for our results, consolidating the language of combinatorial circuits and generalizations of holographic algorithms into the language of category theory.  In Section \ref{sec:DC} we define determinantal circuits and give applications to the rooted spanning forest theorem (Section \ref{ssec:DCApp}), computing Tutte polynomials, and quantum circuits with postselection-based closed timelike curves (P-CTC, Section \ref{ssec:DCAppCTC}). In Section \ref{sec:PfafvsDet} we reframe Pfaffian circuits in terms of monoidal categories and relate determinantal and Pfaffian circuits. 

A remark on notation: in most cases we use $M,N$ for matrices, $\indexset{I},\indexset{J}$ for sets (especially of indices), $f,g$ for morphisms, $\mathbf{F},\mathbf{G}$ for functors, and $\mcC,\mcM$ for categories.  For a matrix $X$, we let $X_{\indexset{IJ}}$ be the submatrix with rows in $\indexset{I}$ and columns in $\indexset{J}$.

\section{Toward a categorical formulation of counting complexity}\label{sec:CCC}

Let $\mcM$ be a (strict) monoidal category ~\cite{maclane1998categories} with monoidal identity $\unit_\mcM$ and such that $S_\mcM=\Hom_\mcM(\unit_\mcM,\unit_\mcM)$ is a semiring; call this a {\em semiringed category}. 
A {\em monoidal word} is a collection of morphisms composed (when domain and codomain match) and tensored together to form a new morphism ~\cite{kassel1995quantum,morton2014belief}. 

\begin{problem}\label{prob:count} 
The  {\em counting problem}  in a semiringed category $\mcM$ is to determine which morphism in $\Hom_\mcM(\unit_\mcM, \unit_\mcM)$ is represented by an arbitrary monoidal word in $\mcM$ with $\unit_\mcM$ as its domain and codomain.   
\end{problem}

Over $\fdVect_\C$, this is sometimes called the {\em tensor contraction problem}.  By demonstrating a case (counting the number of solutions to a $\mathsf{Mon}-3\mathsf{SAT}$ problem) where Problem \ref{prob:count} is known to be $\mathsf{\#P}$-complete, the following example provides a proof that in general Problem \ref{prob:count} is $\mathsf{\#P}$-hard.

\begin{example}
Consider the tensor scheme $\scrT$ ~\cite{selinger2009survey,joyal1991geometry,morton2014belief} with object variable $B$ and morphism variables $v_n\!:\!\unit \ra B^{\ot n}$ for all integers $n >0$ and $c\!:\!B^{\ot 3} \ra \unit$.  The tensor scheme generates a free symmetric  
monoidal category $\scrM_\scrT$
, and a monoidal word in this category is an abstract description of a counting constraint satisfaction problem, where we haven't yet specified details such as which ternary clause will be used.

Now for any semiring $S$ (such as the Boolean semiring or non-negative integers), we can define the semiringed category $S\Rel$ of ``$S$-valued relations'' in which the objects are finite sets and the morphisms are given by functions from the cartesian product of domain and codomain to $S$, that is $\Hom_{S\Rel}(A,B) = \Hom_{\Set}(A \times B,S)$.  
In particular, $\Hom_{S\Rel}(\unit,\unit)=S$.  
In $S\Rel$, monoidal product of objects is cartesian product and the composition of $f\!:\!A\rae B$ and $g\!:\!B \rae C$ is $(g \circ f)(a,c) = \sum_{b \in B} f(a,b) \cdot g(b,c)$, where addition and multiplication are defined in $S$.  An interpretation in such a category is a functor from  $\scrM_\scrT$ to $S\Rel$ specified by assigning values to each object and morphism variable of the tensor scheme.

Now consider an interpretation in $\N \Rel$ assigning the object the value $B=\{0,1\}$.  Write elements of $B^{\ot n}$ as length-$n$ bitstrings. Assign morphisms values by letting $v_n$ be $1$ on $(\unit,0^{\ot n})$ and $(\unit,1^{\ot n})$ and zero otherwise, expressing Boolean variables.  Let $c$ take value $0$ on $(000,\unit)$ and one otherwise, expressing a ternary OR clause.  Then determining which element of $\N$ is represented by each monoidal word $w\!:\!\unit \rae \unit$ is a $\mathsf{Mon}-\#3\mathsf{SAT}$ problem, and this class of monotone counting 3$\mathsf{SAT}$ problems is $\#\mathsf{P}$-complete.  To remove the monotone restriction, we could add a morphism $n:B \ra B$ to the tensor scheme and an appropriate interpretation.




\end{example}

A strict monoidal functor $\mathbf{F}\!:\!\mcM \rae \mcM'$ between semiringed categories is {\em count preserving} if the induced map $\mathbf{F}\!:\!S_\mcM \rae S_{\mcM'}$ is an injective morphism of semirings.  Schemes that generalize holographic algorithms ~\cite{QCtcbSiPT} seek a count-preserving functor from a category in which the counting problem (Problem \ref{prob:count}) is in $\mathsf{FP}$ to a category in which the counting problem is in general $\mathsf{\#P}$-hard.

In each type of circuit, we consider two semiringed categories $\mcC$ and $\mcS$. Let $\mcL$ be a problem of interest.

 We call $\mcC$ the {\em  counting} category and $\mcS$ is a subcategory of $\fdVect_\C$. Then let $i:\mathcal{L}\to\mcC$ be a map that gives an interpretation or encoding of the problem as a string diagram in $\mcC$. By this we mean that for every instance of a problem $l\in\mathcal{L}$, $i(l)$ is a string diagram that solves this instance of the problem.

The category $\mcC$ may have a non-intuitive encoding of the problem but has the advantage that there exists a polynomial-time algorithm to determine which morphism of $S_\mcM=\Hom(\unit,\unit)$ is represented by an arbitrary monoidal word.  We also have an interpretation $f:\mathcal{L}\to\mcS$. Then we want a monoidal functor $h$ such that  the diagram
\[
\xymatrix{
\mcL \ar[r]^{i} \ar[rd]_{f} & \mcC \ar[d]^{h}\\
&\mcS 
}
\]
 commutes and such that the count is preserved by $h$. $\mcS$ is the subcategory generated by the morphisms in the image of either $h \circ i$ or $f$.  The induced maps on $ S_\mcC$ and $S_\mcS$ make $S_\mcC$ a sub-semiring of $S_\mcS$.  The functor $h$ is called $\sDet$, and $\sPf$ for determinantal and Pfaffian circuits respectively in the sequel.



Of course, it is important that the construction represented by the functors is implementable in polynomial time.  Often this is not a concern, because diagrams in $\mcC$ and $\mcL$ are effectively identified, and the problem is expressed in the language that will be used to perform the contraction.

\section{Determinantal circuits}\label{sec:DC}

Suppose $X$ is an $n \times m$ matrix of elements of $\C$ with rows and columns labeled by finite disjoint subsets $\indexset{N}$ and $\indexset{M}$ of $\N=\ZZ_{\geq 0}$.  For $i \in \N$, let $V_i = \C^2$ be spanned by an orthonormal basis (with inner product)
$v_{i,0}, v_{i,1}$ and for finite $\indexset{N} \subset \N$ write $V_\indexset{N} := \ot_{i \in \indexset{N}} V_i$.  
Define the function $\sDet$ (which we later show to be a functor) by $\sDet(\indexset{N})=V_\indexset{N}$ and
\[
\sDet:  \Mat_k (n,m) \ra V_\indexset{N}^* \ot V_\indexset{M} \isom (\C^{2*})^{\ot n} \ot (\C^2)^{\ot m}
\]
\[
\sDet(X) = \sum_{\indexset{I}\subset [n], \indexset{J} \subset [m]} \det(X_{\indexset{IJ}}) \ketbra{ \indexset{I} }{ \indexset{J} }
\]

\noindent where $\ket{\indexset{I}}=\bigotimes_{i \in \indexset{N}} v_{i,\chi(i,\indexset{I})}$, $\bra{\indexset{J}}=\bigotimes_{i \in \indexset{M}} v^*_{i,\chi(i,\indexset{J})}$ and the indicator function $\chi(i,\indexset{I})=0$ if $i \notin \indexset{I}$ and $1$ if $i \in \indexset{I}$. Throughout this paper, we work with the understanding that $\det(X_{\indexset{IJ}})=0$ if $|\indexset{I}|\ne |\indexset{J}|$.



This subdeterminant function $\sDet$ induces a strong monoidal functor $\sDet:\mcC \ra \fdVect_\C$ from a matrix category to a subcategory $\mcD$ of $\fdVect_\C$. 
Let $\mcC$ be the free monoidal category described as follows.  The objects of $\mcC$ are finite ordered subsets of $\N$ (which may have repeated elements), with monoidal product on objects defined by disjoint union.  The morphisms are $\C$-valued matrices with rows and columns labeled by subsets of $\N$.  If $M, N$ are two matrices with the set of row labels of $M$ equal to the set of column labels of $N$, order them and let $N \circ M = NM$ be the ordinary matrix product, with the resulting matrix inheriting the row labels of $N$ and the column labels of $M$.  The monoidal product $\ot_\mcC$ is the direct sum of labeled matrices.

Let $\mcD$ be the image of $\mcC$ in $\fdVect_\C$. It will be the free dagger symmetric traced monoidal subcategory of finite-dimensional $\C$-vector spaces generated by the object $\C^2$, endowed with an orthonormal basis, and morphisms $\sDet(M)$ for $M$ a labeled matrix.  Tensor product and composition/contraction are the usual operations.

Definitions of free dagger symmetric traced monoidal subcategories and related concepts are given in ~\cite{joyal1996traced,selinger2009survey}, and in the proofs below we check the necessary axioms.

\begin{proposition}\label{thm:sdsmc}
 $\mcC$ is a strict dagger symmetric monoidal category.
\end{proposition}
\begin{proof}

In this proof we denote the monoidal product $\otimes_\mcC$ for $\mcC$ defined above simply as $\otimes$. We need to show that it is a bifunctor. For $A\subset\mathbb{N}$, $\id_A$ is the identity matrix with row and column labels $A$. It is easy to see that for any $A,B\subset\mathbb{N}$, $\id_A\otimes\id_B=\id_{A\otimes B}$. Now for morphisms $W,X,Y,Z\in\Mor(\mcC)$, $W\otimes X\circ Y\otimes Z=(W\oplus X)(Y\oplus Z)=WY\oplus XZ=(W\circ Y)\otimes(X\circ Z)$, so $\otimes$ is indeed a bifunctor.

 For $A,B,C\in\Ob(\mcC)$, the associator $\alpha_{ABC}:(A\otimes B)\otimes C\to A\otimes(B\otimes C)$ is just equality by the associativity of matrix direct product. The unit for $\mcC$, denoted $\unit$, is the empty set. Then $\lambda_A:\unit\otimes A\to A$ and $\rho_A:A\otimes\unit\to A$ are also equality since it is union with $\emptyset$. It is clear that $\alpha,\lambda,$ and $\rho$ are natural isomorphisms. 

We need to check that the diagrams from MacLane's Coherence Theorem commute. First let us check, for $A,B\in\text{Ob}(\mcC)$:
\[
\xymatrix{
 (A\otimes \unit)\otimes B \ar[r]^\alpha \ar[rd]_{\rho_A\otimes\id_B} &A\otimes(\unit\otimes B) \ar[d]^{\id_A\otimes\lambda_B}\\
&A\otimes B
}
\]
$(A\otimes \unit)\otimes B=(A\cup\emptyset)\cup B$ is mapped to $A\cup B$ by $\rho_A\otimes\id_B$ via equality. Then $\alpha$ maps $(A\cup \emptyset)\cup B$ to $A\cup(\emptyset\cup B)$ via equality. This is then mapped to $A\cup B$ by $\id_A\otimes\lambda_B$ via equality, and the diagram commutes.

Now let us check the second diagram, for $A,B,C,D\in\Ob(\mcC)$:
\[
 \xymatrix{
& ((C\otimes A)\otimes B)\otimes D \ar[ld]^\alpha \ar[r]_{\alpha\otimes\id_D} & (C\otimes(A\otimes B))\otimes D \ar[dd]_\alpha\\
(C\otimes A)\otimes(B\otimes D) \ar[rd]_\alpha & &\\
& C\otimes(A\otimes(B\otimes D))  &C\otimes((A\otimes B)\otimes D) \ar[l]^{\id_C\otimes\alpha}}.
\]
The object $((C\otimes A)\otimes B)\otimes D)=((C\cup A)\cup B)\cup D$ is mapped to $C\cup(A\cup(B\cup D))$ by $(\id_C\otimes\alpha)\circ(\alpha)\circ(\alpha\otimes\id_D)$ via equality. Similarly, it is mapped to $C\cup(A\cup(B\cup D))$ by $\alpha\circ\alpha$ via equality. This diagram also commutes and so $\mcC$ is a monoidal category. Furthermore, since $\alpha,\lambda$, and $\rho$ are equalities, $\mcC$ is a strict monoidal category.

The braiding for $\mcC$ is a map $c_{A,B}:A\otimes B\to B\otimes A$, $A,B\in\Ob(\mcC)$. It is given by the matrix 
\begin{equation*}
c_{A,B}=\bordermatrix{ & B & A\cr
A &0&1\cr
B &1&0}.
\end{equation*}
We need to check that the following diagrams commute for $A,B,C\in\Ob(\mcC)$:
\[
 \xymatrix{
&(B\otimes A)\otimes C \ar[r]^\alpha & B\otimes(A\otimes C) \ar[rd]^{\id_B\otimes c_{A,C}}&\\
(A\otimes B)\otimes C \ar[ru]^{c_{A,B}\otimes \id_C} \ar[rd]_\alpha & & &B\otimes(C\otimes A)\\
& A\otimes(B\otimes C) \ar[r]^{c_{A,(B\otimes C)}} & (B\otimes C)\otimes A \ar[ru]_{\alpha} &}
\]
\[
 \xymatrix{
&(B\otimes A)\otimes C \ar[r]^\alpha & B\otimes(A\otimes C) \ar[rd]^{\id_B\otimes c^{-1}_{A,C}}&\\
(A\otimes B)\otimes C \ar[ru]^{c^{-1}_{A,B}\otimes \id_C} \ar[rd]_\alpha & & &B\otimes(C\otimes A)\\
& A\otimes(B\otimes C) \ar[r]^{c^{-1}_{A,(B\otimes C)}} & (B\otimes C)\otimes A \ar[ru]_{\alpha} &}.
\] The first diagram commutes by noting that 
\begin{equation*}
\bordermatrix{ & B & A & C\cr
A &0&1&0\cr
B &1&0&0\cr
C &0&0&1}
\bordermatrix{ &B&A&C\cr
B&1&0&0\cr
A&0&1&0\cr
C&0&0&1}
\bordermatrix{ &B&C&A\cr
B&1&0&0\cr
A&0&0&1\cr
C&0&1&0}=
\end{equation*}

\begin{equation*}
\bordermatrix{ &A&B&C\cr
A&1&0&0\cr
B&0&1&0\cr
C&0&0&1}
\bordermatrix{&B&C&A\cr
A&0&0&1\cr
B&1&0&0\cr
C&0&1&0}
\bordermatrix{&B&C&A\cr
B&1&0&0\cr
C&0&1&0\cr
A&0&0&1}.
\end{equation*} The second diagram commutes since $c^{-1}_{B,A}=c_{A,B}$ (which implies the category is symmetric) for any $A,B\in\Ob(\mcC)$ and so the second diagram is the same as the first.

The dagger for $\mcC$ is given by matrix transpose and the identity on objects. 
Clearly $\id_A^\dag=\id_A^T=\id_A$. Given $X,Y\in\Mor(\mcC)$, $X:A\to B$, $Y:B\to C$, $(X\circ Y)^\dag=(XY)^T=Y^TX^T=Y^\dag\circ X^\dag:C\to A$. Lastly $X^{\dag\dag}=X^{TT}=X$. 

We also need the dagger to satisfy two extra properties since we are working in a monoidal category. First, given $X,Y\in\Mor(\mcC)$, $(X\otimes Y)^\dag=(X\oplus Y)^T=X^T\oplus Y^T=X^\dag\otimes Y^\dag$. Secondly, $\alpha,\lambda,$ and $\rho$ should all be unitary (its inverse is equal to its dagger). Since they are all the identity morphism, this is also satisfied. Thus $\mcC$ is indeed a strict dagger symmetric monoidal category.
\end{proof}

\begin{theorem}\label{thm:sdetequiv}
The map $\sDet$ defines a strict monoidal functor which is an equivalence (in fact, an isomorphism) of dagger symmetric traced categories.  Thus while computing a trace in $\fdVect_\C$ is in general $\#P$-hard, in the image of $\sDet$ it can be computed in polynomial time.
\end{theorem}

We prove this in two parts as Lemmata \ref{lem:stdetStrMonEq} and \ref{lem:stdetDagSymEq}.

\begin{lemma}\label{lem:stdetStrMonEq}
The map $\sDet$ defines a strict monoidal functor which is an equivalence of monoidal categories.
\end{lemma}

\begin{proof}
First we must show that $\sDet$ is a functor, i.e.\ that it respects composition and that $\sDet(\id_A)=\id_{\sDet(A)}$. 
Suppose $X \in \Hom_\mcC(\indexset{I},\indexset{J})$, $Y \in \Hom_\mcC(\indexset{J},\indexset{K})$ so $X$ is a matrix with row labels $\indexset{I}$, column labels $\indexset{J}$ and $Y$ has row labels $\indexset{J}$ and column labels $\indexset{K}$: 
\[
\sDet(Y) \circ \sDet(X)=\sum_{\indexset{i}\subseteq \indexset{I}}{\sum_{\indexset{j}\subseteq \indexset{J}}{\sum_{\indexset{k}\subseteq \indexset{K}}{\det(X_{\indexset{ij}})\det(Y_{\indexset{jk}})|\indexset{i}\rangle\langle \indexset{k}|}}}\]\[
=\sum_{\indexset{i}\subseteq \indexset{I}}{\sum_{\indexset{j}\subseteq \indexset{J}}{\det(XY_{\indexset{ik}})|\indexset{i}\rangle\langle \indexset{k}|}}=\sDet(XY)\] 
where the middle equality is the Cauchy-Binet formula. Now in $\mcC$, $\id_A$ is the identity matrix with row and column labels $A$. Then $\sDet(\id_A)=\sum_{I\subseteq A}{|I\rangle\langle I|}$ which is the identity morphism for the object $\sDet(A)$ in $\mcD$, and $\sDet$ is indeed a functor.

For $\sDet$ to be a monoidal functor, we must demonstrate two additional properties.  First we must show that $\sDet(A\oplus B)=\sDet(A)\ot\sDet(B)$. Let $\indexset{I}$ and $\indexset{J}$ be the rows and columns of  $A$, respectively. Let $\indexset{I}'$ and $\indexset{J}'$ be likewise for $B$. A straightforward calculation gives 
$$\sDet(A\oplus B)=\sum_{\indexset{U}\subseteq \indexset{I}\cup \indexset{I}'}{\sum_{\indexset{V}\subseteq \indexset{J}\cup \indexset{J}'}{\det(A\oplus B)_{\indexset{UV}}|\indexset{U}\rangle\langle \indexset{V}|}}$$ 
$$=\sum_{\indexset{U}\subseteq \indexset{I}\cup \indexset{I}'}{\sum_{\indexset{V}\subseteq \indexset{J}\cup \indexset{J}'}{\det(A_{\indexset{U}\cap \indexset{I},\indexset{V}\cap \indexset{J}})\det(B_{\indexset{V}\cap \indexset{I}',\indexset{V}\cap \indexset{J}'})|\indexset{U}\cap \indexset{I}\rangle|\indexset{U}\cap \indexset{I}'\rangle\langle \indexset{V}\cap \indexset{J}|\langle \indexset{V}\cap \indexset{J}'|}}$$ 
$$=\sum_{\indexset{U}\subseteq \indexset{I}}{\sum_{\indexset{U}'\subseteq \indexset{I}'}{\sum_{\indexset{V}\subseteq \indexset{J}}{\sum_{\indexset{V}'\subseteq \indexset{J}'}{\det(A_{\indexset{UV}})\det(B_{\indexset{U}'\indexset{V}'}}}}})|\indexset{U}\rangle|\indexset{U}'\rangle\langle \indexset{V}|\langle \indexset{V}'|=\sDet(A)\ot\sDet(B).$$

Secondly we must show there are morphisms $f_0:\unit_\mcD \to \sDet(\unit_\mcD)$ (the unit in $\mcD$ is the base field $\C$) and for any $A,B\in\Ob(\mcC)$, $f_1:\sDet(A)\otimes \sDet(B)\to \sDet(A\otimes_\mcC B)$ satisfying certain axioms expressed as commutative diagrams.  

Since $\sDet(\emptyset)=\otimes_{i\in\emptyset}{V_i}=\C$, $f_0$ is simply equality. Similarly for objects $A$ and $B$, $$\sDet(A\otimes_\mcC B)=\sDet(A\cup B)=\otimes_{i\in A\cup B}{V_i}$$ $$=(\otimes_{i\in A}{V_i})\otimes(\otimes_{j\in B}{V_j})=\sDet(A)\otimes\sDet(B).$$ Thus $f_1$ is equality. In the following diagrams, we shall call $\sDet$ simply $\mathbf{F}$.

Let $\alpha',\lambda',\rho'$ be the natural transformations for $\mcD$. Note that all three are equalities. For $A,B,C\in\Ob(\mcC)$, the following must commute:
\[
 \xymatrix{
\mathbf{F}(A)\otimes(\mathbf{F}(B)\otimes \mathbf{F}(C)) \ar[r]^{\alpha'} \ar[d]_{\id_{\mathbf{F}(A)}\otimes f_1} & (\mathbf{F}(A)\otimes \mathbf{F}(B))\otimes \mathbf{F}(C) \ar[d]_{f_1\otimes\id_{\mathbf{F}(C)}}\\
\mathbf{F}(A)\otimes(\mathbf{F}(B\otimes_\mcC C)) \ar[d]_{f_1} & (\mathbf{F}(A\otimes_\mcC B)\otimes \mathbf{F}(C)) \ar[d]_{f_1}\\
\mathbf{F}(A\otimes_\mcC (B\otimes_\mcC C)) \ar[r]^{\mathbf{F}(\alpha)} & \mathbf{F}((A\otimes_\mcC B)\otimes_\mcC C)}.
\]

\[
 \xymatrix{
\mathbf{F}(B)\otimes\unit' \ar[r]^{\rho'} \ar[d]^{\id_{\mathbf{F}(B)}\otimes f_0} & \mathbf{F}(B)\\
\mathbf{F}(B)\otimes \mathbf{F}(\unit) \ar[r]^{f_1} & \mathbf{F}(B\otimes\unit) \ar[u]^{\mathbf{F}(\rho)}}\qquad
 \xymatrix{
\unit'\otimes \mathbf{F}(B) \ar[r]^{\lambda} \ar[d]_{f_0\otimes\id_{\mathbf{F}(B)}} & \mathbf{F}(B)\\
\mathbf{F}(\unit)\otimes \mathbf{F}(B) \ar[r]^{f_1} & \mathbf{F}(\unit\otimes B) \ar[u]_{\mathbf{F}(\lambda)}}
\]
The diagrams trivially commute as all of the maps are identities.
So $\sDet$ is a strong monoidal functor. Since $f_0,f_1$ are equalities, it is a strict monoidal functor.

Lastly, we want to say that $\mcC$ and $\mcD$ are equivalent as monoidal categories. By definition of $\mcD$, $\sDet$ surjects onto objects and morphisms, so it is a full functor. Now consider $\Hom(A,B)$ for objects $A,B\in\Ob(\mcC)$. Let $X\in\Hom(A,B)$. $\sDet(X)$ contains all the entries of $X$ as coefficients in the sum since the entries of $X$ are $1\times 1$ minors, and  $X$ is determined by its image $\sDet(X)$. Thus $\sDet$ induces an injection on $\Hom(A,B)\to\Hom(\sDet(A),\sDet(B))$, and the functor is faithful. Thus it is an equivalence. However, it is not quite an isomorphism as $\sDet$ does not give a bijection on objects as all subsets of $\N$ of size $n$ map to $(\C^2)^{\ot n}$.
\end{proof}

We have yet to define the braiding and dagger for $\mcD$ required to state Theorem \ref{thm:sdetequiv}. For $\mathbf{F}=\sDet$ to respect the braiding,  we need the following diagram to commute:
\[
\xymatrix{
 \mathbf{F}(A)\otimes \mathbf{F}(B) \ar[r]^{f_1} \ar[d]_{c_{\mathbf{F}(A),\mathbf{F}(B)}} & \mathbf{F}(A\ot B) \ar[d]^{\mathbf{F}(c_{A,B})}\\
\mathbf{F}(B)\otimes \mathbf{F}(A) \ar[r]^{f_1} & \mathbf{F}(B\ot_\mcC A)}.
\]   Recalling the matrix $c_{A,B}$ as defined in Theorem \ref{thm:sdsmc}, we define the braiding for $\mcD$ to be $F(c_{A,B})=\sDet(c_{A,B})=|00\rangle\langle 00|+|01\rangle\langle10|+|10\rangle\langle01|-|11\rangle\langle11|$, which makes the diagram commute trivially. We do not check the diagrams that ensures this is a valid braiding for $\mcD$ since it is equivalent to $\mcC$. For the dagger, consider $X\in\Mor(\mcC)$ with row labels $\mathsf{I}$ and column labels $\mathsf{J}$, and note 
$$\sDet(X^\dag)=\sum_{\mathsf{i}\subseteq \mathsf{I},\mathsf{j}\subseteq \mathsf{J}}{\det(X_{\mathsf{i}\mathsf{j}}^T)|\mathsf{i}\rangle\langle \mathsf{j}|}=$$
$$\sum_{\mathsf{i}\subseteq \mathsf{I},\mathsf{j}\subseteq \mathsf{J}}{\det(X_{\mathsf{ij}})|\mathsf{j}\rangle\langle \mathsf{i}|}=\sDet(X)^T.$$ So the dagger for $\mcD$ is the normal dagger in $\fdVect_k$.

For $f:A \ra A\in \mcC$, define $\tr(f) = \det( I + f)$ and define trace in $\mcD$ in the usual way. This choice of trace may seem unusual, but it satisfies the axioms of a traced category ~\cite{selinger2009survey} and its image under the $\sDet$ functor is the usual trace in $\mcD$ (as we show in a moment). This is the most important aspect as it allows us to frame problems in $\mcC$ and find the answer to the contraction problem without the need to pass over to the category $\mcD$ which has exponentially larger tensors.

\begin{lemma}\label{lem:stdetDagSymEq}
The map $\sDet$ defines a strict monoidal functor which is an equivalence of dagger symmetric traced categories.
\end{lemma}

\begin{proof}
By construction, $\sDet$ respects the braiding.  We also showed that this functor respects the normal dagger for linear transformations. Theorem \ref{thm:detevaltime} and Proposition \ref{prop:detkernel} below shows that $\sDet$ induces the identity map from $\Hom(\unit_\mcC,\unit_\mcC)\to\Hom(\unit_\mcD,\unit_\mcD)$ and thus respects the trace.

\end{proof}

\begin{remark}
This braiding is {\em not} the usual braiding for $\Vect_\C$.  Thus while the functor $\sDet$ is count-preserving, the count will not be the same as if the standard braiding $u \otimes v \mapsto v \otimes u$ is used. 
\end{remark}


Using the operations of $\oplus$ and matrix multiplication, we can transform any string diagram in $\mcC$ into a diagram with a single matrix, $M$, and thus evaluate the determinantal circuit efficiently. 

A determinantal circuit is the trace of a linear map defined by an expression of the form $(f_{1,1} \otimes \cdots \otimes f_{1,n_1}) \circ \cdots \circ (f_{m,1} \otimes \cdots \otimes f_{m,n_m})$. Let $d_k$ be the dimension of the domain of the $k$th linear map $(f_{k,1} \otimes \cdots \otimes f_{k,n_k})$, with $k=1,\dots,m$. The maximum width of such a circuit is $\max_{k=1,\dots,m}\log_2d_k$ and the depth is $m$.

\begin{theorem}\label{thm:detevaltime}
The time complexity of computing the trace of a determinantal circuit in $\mcC$ is $O(dw^\omega) = O(dw^\omega + c^\omega)$ where $d$ is the depth of the circuit, $w$ is the maximum width, $c$ is width at the input and output (so can be chosen to be the minimum width), and $\omega$ is the exponent of matrix multiplication.
\end{theorem}

\begin{proof}
We have an $n\times n$ matrix with equal row and column labels, which we may assume to be $1,\dots,n$. Then $$\sDet(M)=\sum_{\mathsf{I},\mathsf{J}\subseteq[n]}{\det(M_{\mathsf{I},\mathsf{J}})|\mathsf{I}\rangle\langle \mathsf{J}|}$$ and contracting this against itself gives $$\sum_{\mathsf{I},\mathsf{J}\subseteq[n]}{\det(M_{\mathsf{I},\mathsf{J}})\langle \mathsf{J}|\mathsf{I}\rangle\langle \mathsf{J}|\mathsf{I}\rangle}=\sum_{\mathsf{I} \subset [n]} \det M_{\mathsf{I},\mathsf{I}}.$$ That is, the trace of a matrix $M$ in $\mcC$ is the exponentially large sum of its $2^n$ principal minors; we claim that $\det(I+A)$ is precisely this sum (Proposition \ref{prop:detkernel}).  This enables us to compute this number in time $n^\omega$. 
\end{proof}

The following identity is well-known (e.g.\ it can be derived from results in ~\cite{horn1990matrix}); we include a proof for completeness.

\begin{proposition}\label{prop:detkernel}
Given an $n\times n$ matrix $M$, $$\det(I+M)=\sum_{\mathsf{J}\subseteq [n]}{\det(M_\mathsf{J})}$$
 where $M_{\mathsf{J}}=M_{\mathsf{J},\mathsf{J}}$.
\end{proposition}
\begin{proof}

Let $u_i$ be the columns of $M$ and $e_\mathsf{i}$ the standard basis vectors, $\mathsf{i}\in[n]$. Then $\det(I+M)=\bigwedge_{\mathsf{i}=1}^{n}{(e_\mathsf{i}+u_\mathsf{i})}$. Expanding this gives the sum of the determinants of all $2^n$ matrices with $i$th column either $u_\mathsf{i}$ or $e_\mathsf{i}$. 

Consider one of these matrices, $W$. Let $\mathsf{J}\subseteq[n]$ be the set of indices of the $u_\mathsf{j}$ appearing as columns in $W$. Then for any $\mathsf{j}\notin \mathsf{J}$, $e_\mathsf{j}$ is a column of $W$. Using the Laplace expansion, $\det(W)=\det(W_{\overline{\mathsf{j}}})$, where $W_{\overline{\mathsf{j}}}$ is $W$ with the $\mathsf{j}$th row and column omitted. Then iterating the Laplace expansion gives us that $\det(W)=\det(M_{\mathsf{J}})$.
\end{proof}

A monoidal category is said to {\it have duals for objects} or be {\it closed} if each object $A$ has a dual object $A^*$ related by an adjunction $(A,A^*, i_A, e_A)$. 
Note that while  $\mcD$ could be equipped with the object duality structure $(A,A^*, i_A, e_A)$ from the category of finite-dimensional vector spaces to obtain a dagger closed compact category, the matrix category $\mcC$ is {\em not} a closed compact category: it lacks the morphisms $i_A$ (coevaluation) and $e_A$ (evaluation). The morphism $e_A:A \ot A^* \ra I$ would have to be the $\sDet$ of a $2 \times 0$ matrix, or the composition of several morphisms to obtain one of this type.


\begin{proposition}
The category $\mcC$ does not have duals for objects.
\end{proposition}
\begin{proof}
We cannot have $e_A = \sDet(M)$ for any $M$.  The morphism we want is $\ket{00}+ \ket{11}$, but there is a unique $2 \times 0$ matrix $M$ and $\sDet(M)=\ket{00}$.

\end{proof}
As a consequence, we really do have to work with traced categories rather than the more convenient dagger closed compact categories ~\cite{joyal1996traced}.

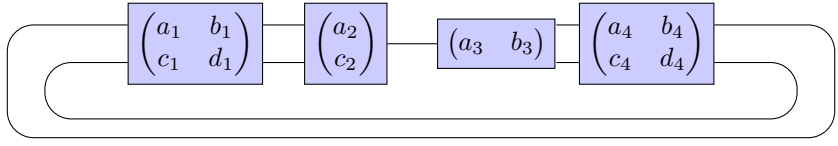
\begin{figure}
\begin{tikzpicture}
[Atensor/.style={rectangle,draw, fill=blue!20}]
\node (a0) at (0,0) [Atensor] {$\begin{pmatrix}a_1&b_1\\c_1&d_1 \end{pmatrix}$};
\draw (.9,.25) -- (2,.25);
\draw (.9,-.25) -- (2,-.25);
\node (a1) at (2,0) [Atensor] {$\begin{pmatrix}a_2\\c_2 \end{pmatrix}$};
\draw (2.55,0) -- (3.5,0);
\node (a2) at (4,0) [Atensor] {$\begin{pmatrix}a_3&b_3 \end{pmatrix}$};
\draw (4.8,.25) -- (6,.25);
\draw (4.8,-.25) -- (6,-.25);
\node (a3) at (6,0) [Atensor] {$\begin{pmatrix}a_4&b_4\\c_4&d_4 \end{pmatrix}$};
\draw[rounded corners=10pt] (6.9,-.25) -- (8,-.25) -- (8,-1) -- (-2,-1) -- (-2,-.25) -- (-.9,-.25);
\draw[rounded corners=10pt] (6.9,.25) -- (8.5,.25) -- (8.5,-1.25) -- (-2.5,-1.25) -- (-2.5,.25) -- (-.9,.25);
\end{tikzpicture}
\caption{An example of a determinantal circuit (wires oriented clockwise).  The four tensors in $\Vect_\C$, from left to right, are obtained by applying  $\sDet$ to each matrix.  Letting $V\!=\!\C^2$, they lie in $(V^*)^{\ot 2} \otimes V^{\ot 2}$, $(V^*)^{\ot 2} \otimes V$, $V^* \otimes V^{\ot 2}$, and $(V^*)^{\ot 2} \otimes V^{\ot 2}$ respectively.} \label{fig:detcirex}
\end{figure}

A diagram in the equivalent categories $\mcC, \mcD$ is called a {\em determinantal circuit}, an example is given in Figure \ref{fig:detcirex}. When the morphism represented is a field element, it computes the partition function, i.e. counts the weighted number of solutions to the weighted counting constraint satisfaction problem it represents.  Because these categories have a traced, dagger braided monoidal category structure, they come with a corresponding graphical language ~\cite{selinger2009survey}. 

It is also a question of interest which tensors are determinantal. One can test whether a vector can be the set of determinants of minors from a matrix using the Pl\"ucker relations to obtain the relations among general minors of matrices.  On the other hand, for minors of a fixed size this is an open problem ~\cite{bruns2011relations}.

\section{Applications} \label{ssec:DCApp}

\subsection{Multicycles}

We now discuss a diagrammatic language and describe what determinantal circuits count in terms of \emph{multicycles}.  Our aim is 
to facilitate the application of determinantal circuits to specific counting problems.

Our convention shall be that tensors will be composed from right to left and that tensoring will be from top to bottom. A determinantal circuit is given as the trace of a composition of linear maps $(f_{1,1} \otimes \cdots \otimes f_{1,n_1}) \circ \cdots \circ (f_{m,1} \otimes \cdots \otimes f_{m,n_m})$. Let $S_i=f_{i,1}\ot\cdots\ot f_{i,n_i}$.  Let $M^{S_i}$ be the matrix such that $\sDet(M^{S_i})=S_i$.  We call the $S_i$ or associated $M^{S_i}$ \emph{stacks}.  Pictorially, the situation is as follows:

\begin{center}
\begin{tikzpicture}
\draw (0,0) rectangle (.75,.8);
\draw (.375,.4) node{$f_{1,1}$};
\draw (-.35,.1) -- (0,.1);
\draw (-.1,.5) node{$\vdots$};
\draw (-.35,.7) -- (0,.7);
\draw (.75,.1) -- (1.1,.1);
\draw (.85,.5) node{$\vdots$};
\draw (.75,.7) -- (1.1,.7);
\draw[dashed] (-.25,-1.7) rectangle (1,1);
\draw (.4,1.25) node{$S_1$};
\draw (-.35,.1) arc (90:270:1.7);
\draw (-.35,.7) arc (90:270:2.3);

\draw (.375,-.25) node{$\vdots$};

\draw (0,-.75) rectangle (.75,-1.55);
\draw (.42,-1.15) node{$f_{1,n_1}$};
\draw (-.35,-1.45) -- (0,-1.45);
\draw (.75,-1.45) -- (1.1,-1.45);
\draw (-.35,-.85) -- (0,-.85);
\draw (.75,-.85) -- (1.1,-.85);
\draw (-.1,-1.05) node{$\vdots$};
\draw (.85,-1.05) node{$\vdots$};
\draw (-.35,-1.45) arc (90:270:.2);
\draw (-.35,-1.85) -- (4.35,-1.85);
\draw (-.35,-.85) arc (90:270:.8);
\draw (-.35,-2.45) -- (4.35,-2.45);

\draw (2,-.35) node{$\cdots$};
\draw (2,-2.75) node{$\vdots$};

\draw (3,0) rectangle (4,.8);
\draw (3.5,.4) node{$f_{m,1}$};
\draw (2.65,.1) -- (3,.1);
\draw (2.9,.5) node{$\vdots$};
\draw (2.65,.7) -- (3,.7);
\draw (4,.1) -- (4.35,.1);
\draw (4.1,.5) node{$\vdots$};
\draw (4,.7) -- (4.35,.7);
\draw[dashed] (2.75,-1.7) rectangle (4.25,1);
\draw (3.4,1.25) node{$S_m$};
\draw (4.35,.1) arc (90:-90:1.7);
\draw (4.35,.7) arc (90:-90:2.3);
\draw (-.35,-3.3) -- (4.35,-3.3);
\draw (-.35, -3.9) -- (4.35,-3.9);

\draw (3.375,-.25) node{$\vdots$};

\draw (3,-.75) rectangle (4,-1.55);
\draw (3.55,-1.15) node{$f_{m,n_m}$};
\draw (2.65,-1.45) -- (3,-1.45);
\draw (4,-1.45) -- (4.35,-1.45);
\draw (2.65,-.85) -- (3,-.85);
\draw (4,-.85) -- (4.35,-.85);
\draw (2.9,-1.05) node{$\vdots$};
\draw (4.1,-1.05) node{$\vdots$};
\draw (4.35,-1.45) arc (90:-90:.2);
\draw (4.35,-.85) arc (90:-90:.8);
\end{tikzpicture}.
\end{center}
Forgetting, for a moment, the categorical structure of the circuit, we consider the above as a graph.
\begin{defn}
 A \emph{multicycle} of a graph is an edge-disjoint union of cycles in the graph. We consider the empty graph a multicycle.
\end{defn}
We are interested in whether a subgraph can be interpreted as several cycles, not which edges are in which particular cycles. 
Call two multicycles equivalent if they contain the same edges, and denote an equivalence class of multicycles by $[\mathscr{C}]$.

\begin{defn}
 A \emph{weighted multicycle} of a determinantal circuit is a multicycle of the underlying graph where each cycle in the multicycle is assigned a scalar. The weight of the multicycle is the product of these scalars.
\end{defn}

\begin{proposition}
 Given a determinantal circuit, let $\mathscr{M}$ be the set of all equivalence classes of its multicycles. There exists an assignment of a weight $W_{[\mathscr{C}]}$ to every $[\mathscr{C}]\in\mathscr{M}$ such that the value of the determinantal circuit is $\sum_{[\mathscr{C}]\in\mathscr{M}}{W_{[\mathscr{C}]}}$.
\end{proposition}
\begin{proof}
 
A determinantal circuit with a single $n\times n$ matrix $M$ has value $$\det(I+M)=\sum_{\indexset{I}\subseteq[n]}{\det(M_\indexset{I})} =\Tr(\sDet(M))=\Tr\bigg(\sum_{\indexset{I}\subseteq[n]}{\det(M_\indexset{I})|\indexset{I}\rangle\langle\indexset{I}|}\bigg).$$

A general determinantal circuit is the trace of a composition of stacks $S_1\circ\cdots\circ S_m$. Let $E_k$ be the set of edges entering $S_k$ from the left 
and 
exiting $S_{k-1}$ to the right, and observe that
 $$\Tr(\sDet(M^{S_1}\circ\cdots\circ M^{S_m}))=\Tr\bigg(\sum_{ \indexset{I}_k\subseteq E_k}{\prod{\det(M_{\indexset{I}_k}^{S_k})}|\indexset{I}_1\rangle\langle\indexset{I}_2|\indexset{I}_2\rangle\cdots\langle\indexset{I}_m|\indexset{I}_m\rangle\langle\indexset{I}_{1}|}\bigg)$$
\begin{equation} \label{eq:sumofprod}
=\sum_{\indexset{I_k}\subseteq E_k}{\prod{\det(M_{\indexset{I}_k}^{S_k})}}.
\end{equation}
 
 We want to describe \eqref{eq:sumofprod} as a sum over equivalence classes of multicycles of $S_1\circ\cdots\circ S_m$. 
Consider 
the subgraph of the determinantal circuit whose edges are those in the sets $\indexset{I}_k$. 
We claim that if the subgraph does not correspond to an equivalence class of multicycles, $\prod{\det(M_{\indexset{I}_k}^{S_k})}=0$.
 
Each summand $\prod{\det(M_{\indexset{I}_k}^{S_k})}$ in \eqref{eq:sumofprod} will be non-zero only if $|\indexset{I}_1|=\cdots=|\indexset{I}_m|$ as the determinant of a non-square matrix is zero.This implies that the number of edges of a entering a vertex from the left in the underlying graph must equal the number of edges exiting it to the right. 
This is sufficient for the circuit subgraph given by the subsets $\indexset{I}_k$ to be viewable as a multicycle. 

We have not specified 
a cycle decomposition of the multicycle
, so each circuit subgraph 
represents an equivalence class of multicycles with weight $\prod{\det(M_{\indexset{I}_k}^{S_k})}$.
\end{proof}

\begin{example}
Suppose we are given the following determinantal circuit:
\begin{center}
\begin{tikzpicture}
 [Atensor/.style={rectangle,draw, fill=none}]
\node (a0) at (0,0) [Atensor] {$\begin{pmatrix}a&b\\c&d \end{pmatrix}$};
\draw (-1,-.2) -- (-.75,-.2);
\draw (-1,.2) -- (-.75,.2);
\draw (.75,-.2) -- (1,-.2);
\draw (.75,.2) -- (1,.2);
\draw (-1,-.2) arc (90:270:.4);
\draw (-1,.2) arc (90:270:.8);
\draw (1,.2) arc (90:-90:.8);
\draw (1,-.2) arc (90:-90:.4);
\draw (-1,-1) -- (1,-1);
\draw (-1,-1.4) -- (1,-1.4);
\end{tikzpicture}. 
\end{center}
Its value is the sum of the principal minors of the matrix: $1+a+d+ad-bc$. In the picture below we draw the weighted multicycles in bold on the circuit:
\begin{center}
\begin{tikzpicture}
 [Atensor/.style={rectangle,draw, fill=none}]
\node (a0) at (0,0) [Atensor] {$\begin{pmatrix}a&b\\c&d \end{pmatrix}$};
\draw[dashed] (-1,-.2) -- (-.75,-.2);
\draw[dashed] (-1,.2) -- (-.75,.2);
\draw[dashed] (.75,-.2) -- (1,-.2);
\draw[dashed] (.75,.2) -- (1,.2);
\draw[dashed] (-1,-.2) arc (90:270:.4);
\draw[dashed] (-1,.2) arc (90:270:.8);
\draw[dashed] (1,.2) arc (90:-90:.8);
\draw[dashed] (1,-.2) arc (90:-90:.4);
\draw[dashed] (-1,-1) -- (1,-1);
\draw[dashed] (-1,-1.4) -- (1,-1.4);
\draw (0,1) node{weight=1};

\draw (4,1) node{weight=$a$};
\node (a0) at (4,0) [Atensor] {$\begin{pmatrix}a&b\\c&d \end{pmatrix}$};
\draw[dashed] (3,-.2) -- (3.25,-.2);
\draw[thick] (3,.2) -- (3.25,.2);
\draw[dashed] (4.75,-.2) -- (5,-.2);
\draw[thick] (4.75,.2) -- (5,.2);
\draw[dashed] (3,-.2) arc (90:270:.4);
\draw[thick] (3,.2) arc (90:270:.8);
\draw[thick] (5,.2) arc (90:-90:.8);
\draw[dashed] (5,-.2) arc (90:-90:.4);
\draw[dashed] (3,-1) -- (5,-1);
\draw[thick] (3,-1.4) -- (5,-1.4);

\end{tikzpicture}
\end{center}
\begin{center}
\begin{tikzpicture}
 [Atensor/.style={rectangle,draw, fill=none}]
\node (a0) at (0,0) [Atensor] {$\begin{pmatrix}a&b\\c&d \end{pmatrix}$};
\draw[thick] (-1,-.2) -- (-.75,-.2);
\draw[dashed] (-1,.2) -- (-.75,.2);
\draw[thick] (.75,-.2) -- (1,-.2);
\draw[dashed] (.75,.2) -- (1,.2);
\draw[thick] (-1,-.2) arc (90:270:.4);
\draw[dashed] (-1,.2) arc (90:270:.8);
\draw[dashed] (1,.2) arc (90:-90:.8);
\draw[thick] (1,-.2) arc (90:-90:.4);
\draw[thick] (-1,-1) -- (1,-1);
\draw[dashed] (-1,-1.4) -- (1,-1.4);
\draw (0,1) node{weight=$d$};

\draw (4,1) node{weight=$ad-bc$};
\node (a0) at (4,0) [Atensor] {$\begin{pmatrix}a&b\\c&d \end{pmatrix}$};
\draw[thick] (3,-.2) -- (3.25,-.2);
\draw[thick] (3,.2) -- (3.25,.2);
\draw[thick] (4.75,-.2) -- (5,-.2);
\draw[thick] (4.75,.2) -- (5,.2);
\draw[thick] (3,-.2) arc (90:270:.4);
\draw[thick] (3,.2) arc (90:270:.8);
\draw[thick] (5,.2) arc (90:-90:.8);
\draw[thick] (5,-.2) arc (90:-90:.4);
\draw[thick] (3,-1) -- (5,-1);
\draw[thick] (3,-1.4) -- (5,-1.4);

\end{tikzpicture}.
\end{center}
\end{example}

\subsection{Recovering the matrix tree theorem}  \label{sec:ChungLanglands}

\begin{figure} 
 \centering
\subfloat[A rooted graph]{\includegraphics[scale=.2]{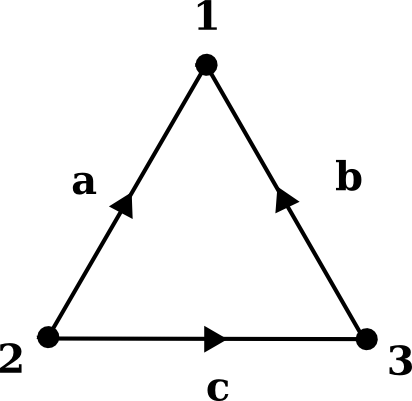}}\quad
\subfloat[Its corresponding circuit]{\includegraphics[scale=.25]{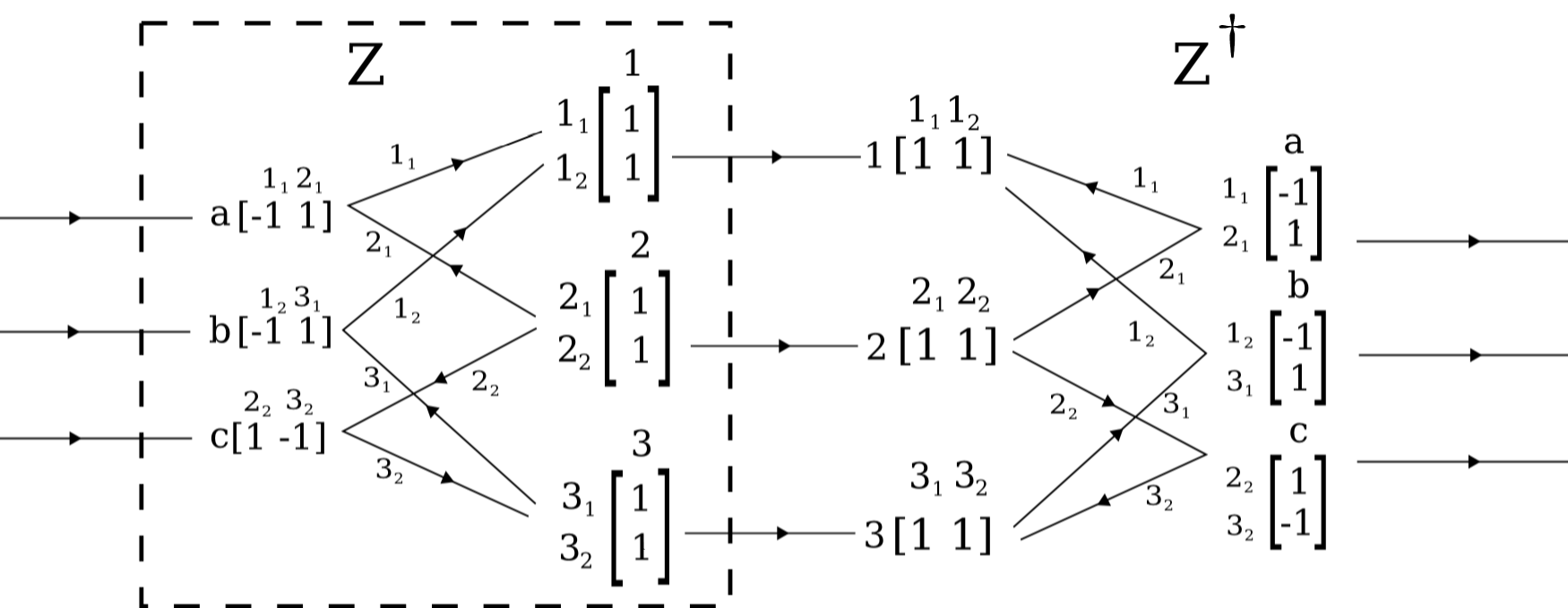}}\quad
\subfloat[Equivalent circuit]{\includegraphics[scale=.3]{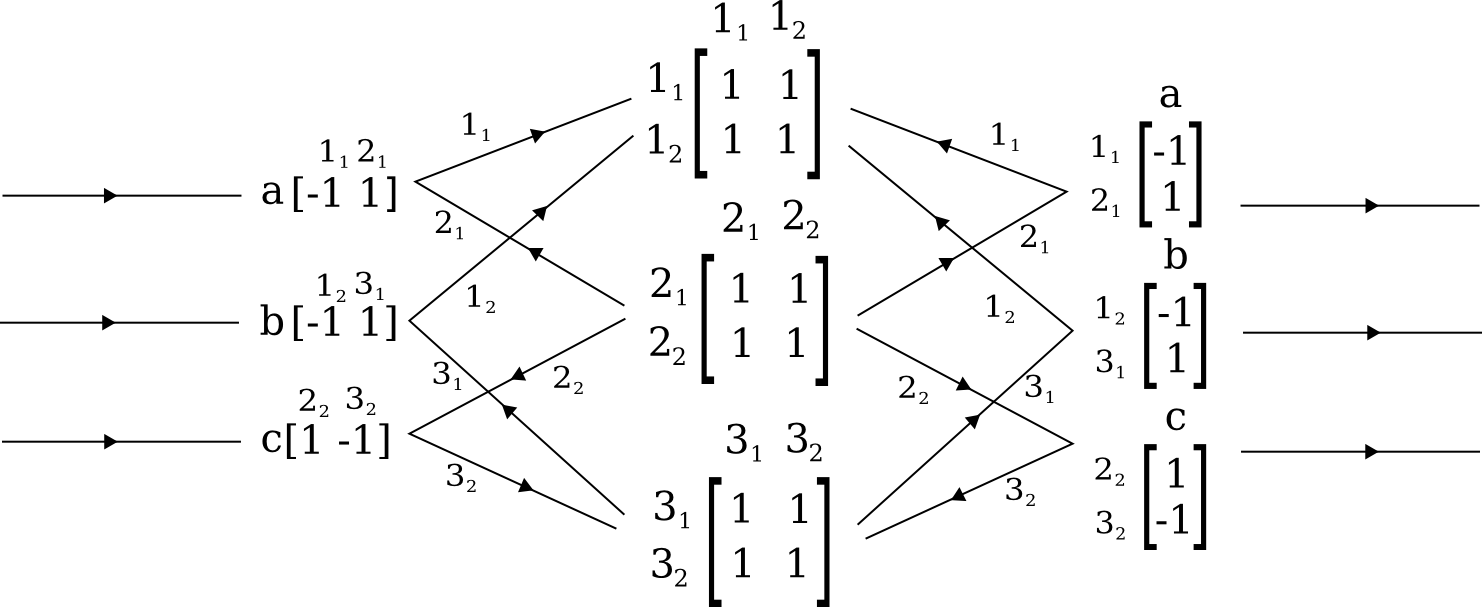}}
\caption{Transforming a rooted graph to a determinantal circuit.}\label{fig:threecircuits}
\end{figure}

One of the main inspirations of determinantal circuits was the rooted spanning form theorem.
\begin{theorem}[~\cite{MR1401006}]\label{thm:chunglanglands} Given a graph $G$, let $B$ be its incidence matrix endowed with an arbitrary orientation. Then $\det(I+BB^T)$ is the number of rooted spanning forests.
\end{theorem}

Recalling the discussion from Section \ref{sec:CCC}, our problem $\mcL$ is to count the number of rooted spanning forests in a graph. Since our problem is in $\#P$, there is a known way to encode the problem as any $\#P$-complete problem, e.g. a $\#SAT$ problem, which can be easily turned into a tensor contraction problem in $\fdVect_\C$ giving a map $f$.

The map $f$ represents the na\"ive way of turning the problem $\mcL$ into a contraction problem. If we use the $\sDet$ functor, this gives a map from $\C$ to $\fdVect_\C$, and we wish to construct an interpretation map $i$ from $\mcL$ into $\mcC$ such that the following diagram commutes:
\[
\xymatrix{
\mcL \ar[r]^{i} \ar[rd]_{f} & \mcC \ar[d]^{h}\\
&\mcS 
}
\]
 where $\mcS=\Hom(\unit,\unit)=\C$ in $\fdVect_\C$. The map $i$ is not obvious, so we construct explicitly.

 We construct a string diagram $ZZ^{\dagger}$ in $\mcC$ which can be reduced to a determinantal circuit consisting of only the matrix $BB^T$ using the operations of $\oplus$ and matrix multiplication. An example of a graph is given in Figure \ref{fig:threecircuits}(a) and the determinantal circuit constructed for it in Figure \ref{fig:threecircuits}(b).
%

Choose an arbitrary orientation on the given graph $G=\{V,E\}$. We first build a string diagram, $Z$, from a collection of $\mcC$-morphisms (nodes); there is one node for every edge and vertex of $G$. Denote an edge of $G$ by $\epsilon$, the edge node in $Z$ corresponding to it by $e$ and the edge node in $Z^{\dagger}$ corresponding to it by $e^{\dagger}$. Denote a vertex in $G$ by $\nu$ and its node in $D_G$ by $v$. An edge node is connected to a vertex node if the edge and vertex are incident in $G$. 

Define an orientation on $Z$ which has no categorical meaning, but is used in the proof. An wire in $Z$ connecting an edge and vertex node is oriented towards the vertex node if that vertex is a sink for the edge in $G$; otherwise the wire is oriented towards the edge node.  Arrange $Z$ into two stacks: the first consists of the edge nodes, the second of the vertex nodes. The dashed box in Figure \ref{fig:threecircuits}(b) gives an example of this construction.

Edge nodes are $1\times 2$ matrices, vertex nodes are $d(\nu)\times 1$ matrices, where $d(\nu)$ is the degree of $\nu$.  The matrix $M_e$ associated to an edge node $e$ in $Z$ is either $[1\;-1]$ or $[-1\;1]$; it  has a $-1$ in the column corresponding to the output wire oriented away from $e$ and a $1$ in the other column.  Let $v$ be a vertex node. The matrix $M_v$ associated with a vertex node $v$ is a $d(\nu)\times 1$ matrix with every entry equal to $1$. 
Although in general we suppress it in pictures, whenever two wires cross, we put the braiding matrix $\begin{pmatrix}0&1\\1&0\end{pmatrix}$ on the crossing.

\begin{lemma}
 Using the operations of matrix multiplication and $\oplus$, the matrices in $Z$ collapse to the incidence matrix of $G$ with some orientation placed on it.
\end{lemma}
\begin{proof}
 
Let $E$ be the matrix equal to 
the direct sum of all the matrices on the edge nodes and $V$ be the direct sum of all the matrices on vertex nodes. Then $Z$ reduces to the matrix $A=EPV$ where $P$ is the permutation matrix obtained from crossed wires. Let $e$ be an edge node and let $r_e$ be the row vector of $E$ corresponding to $e$. For any column vector $c_v$ of $PV$ associated with vertex node $v$, $r_e\cdot c_v\ne 0$ if and only if $e$ is incident to $v$. In fact, $r_e\cdot c_v$ is equal to the number of wires $v\to e$ minus the number of wires $e\to v$ in $Z$. This implies that $A=B$, the incidence matrix.

%
%
%
%

\end{proof}

Reflect $Z$ across a vertical line, transposing all node matrices, to obtain $Z^{\dagger}$, which collapses to the matrix $B^T$. Our final circuit $ZZ^{\dagger}$ is the composition of $Z$ with $Z^{\dagger}$. Figures \ref{fig:threecircuits}(a) and (b), show an example of a graph $G$ and its transformation into a circuit $ZZ^{\dagger}$. We denote the determinantal circuit like in Figure \ref{fig:threecircuits}(b) associated to a graph $G$, $D_G$. Then our map $i$ sends $G\to D_G$.

By analyzing the values of the multicycles of $D_G$ and what they represent in the graph $G$, one can arrive at Theorem \ref{thm:chunglanglands}, although the proof via this method is quite tedious. We next look at another example that is a bit more naturally suited to computation by determinantal circuits.

\subsection{Computing the Tutte Polynomial of Lattice Path Matroids}

Lattice path matriods are a particulary nice and frequently occuring class of matriods whose independent sets are determined by monotone paths on some bounded subset of the integer lattice ~\cite{bonin2003lattice}. Their Tutte polynomials can be calculated in polynomial time ~\cite{bonin2007multi} and can be given a combinatorial interpretation in terms of monotone paths. 

As evidenced at the beginning of this section, determinantal circuits are particulary well suited for computing the Tutte polynomials of these matroids. In fact, using determinantal circuits, an improvement to the algorithm over the algorithm given in ~\cite{bonin2007multi}. It had been noticed previously that Pfaffian circuits also were well suited to this problem ~\cite{morton2010pfaffian}. However, in the following section, we show that the algorithm given by Pfaffian circuits does not constitute an improvement over the original algorithm.

Let $P$ and $Q$ be two monotone paths from $(0,0)$ to $(m,r)$ with $P$ never going above $Q$. More precisely, there are no points $(p_1,p_2)\in P$ and $(q_1,q_2)\in Q$ such that $p_1-q_1<0$ or $p_2-q_2>0$. Now consider the region of $\mathbb{Z}^2$ bounded by (and including) $P$ and $Q$. Let $P=p_1p_2\cdots p_{m+r}$ be the steps of $P$ where $p_i$ either a single step north or a single step east. Let $Q=q_1q_2\cdots q_{m+r}$ be likewise.

\begin{defn}
 Let $\{p_{u_1},\dots,p_{u_r}\}$ be the set of north steps of $P$ and $\{q_{\ell_1},\dots,q_{\ell_r}\}$ be likewise. Define $N_i$ to be the interval $[\ell_i,u_i]$. Then define $M[P,Q]$ to be the matroid with ground set $[m+r]$ and independent sets $N_i$, $i\in[r]$. A \emph{lattice path matriod} is any matroid isomorphic to $M[P,Q]$ for some $P$ and $Q$ as described.
\end{defn}

We consider the region of $\mathbb{Z}^2$ bounded by $P$ and $Q$ as graph $G[P,Q]$ where two points are connected if the differ by $(\pm 1,0)$ or $(0,\pm 1)$. The Tutte polynomial of the matroid $M[P,Q]$ can be thought of as the sum over weighted paths in this graph. 

We can associate to each edge in $G[P,Q]$ a scalar or variable. The weight of a path is simply the product of the weights of its edges. We define $\omega(G[P,Q])=\sum{w(p)}$ where $p$ ranges over the monotone paths in $G[P,Q]$ from $(0,0)$ to $(m,r)$ and $w(p)$ is its corresponding weight.

\begin{theorem}[~\cite{bonin2003lattice}]
 The Tutte polynomial of a lattice path matroid $M[P,Q]$ is $\omega(G[P,Q])$ where the weighting of $G[P,Q]$ is such: the north steps of $Q$ have weight $x$, the east steps of $P$ having weight $y$, and all other weights are 1.
\end{theorem}
\begin{theorem}[~\cite{bonin2007multi}]
 The Tutte polynomial of $M[P,Q]$ can be computed in time $O(n^5)$, where $n=m+r$.
\end{theorem}

This can easily be set up as a determinantal circuit. We use the graph $G[P,Q]$ along with another wire connecting $(0,0)$ and $(m,r)$ which we will give weight 1. We simply need to specify the the matrices that we place on each vertex. Given a vertex $v$, we denote its western and southern	 wires as incoming and its other wires as outgoing. A vertex may of course be missing some of these wires. If $M_v$ is the matrix we place on $v$, all of the entries of a column of $M_v$ is the weight of the corresponding outgoing edge. We denote this determinantal circuit $D_{G[P,Q]}$. In this example, our interpretation map $i$ is almost the identity.

\begin{theorem}[~\cite{morton2013computing}]
 The value of the determinantal circuit $D_{G[P,Q]}$ is the Tutte polynomial of $M[P,Q]$ and this gives an algorithm with running time $O(n^4)$ where $n=m+r$.
\end{theorem}

\subsection{Simulating quantum circuits in the presence of closed timelike curves}\label{ssec:DCAppCTC}
Determinantal circuits define a class of tensor networks with a polynomial-time contraction algorithm.  An immediate consequence is that certain types of quantum circuits (or more generally tensor networks possibly including preparations and postselection)  can be simulated efficiently using this technique.  Essentially these are the tensor networks of the type shown in Figure \ref{fig:detcirex} (with arbitrarily many wires and transformations). 

The loop in such a circuit corresponds to a postselected closed timelike curve (P-CTC) ~\cite{lloyd2011closed}. 
The resulting logical category of circuits represent physical experiments (which, if they contain an embedded contradiction, have count zero ~\cite{morton2012undecidability}).

\section{Relation to Pfaffian circuits}\label{sec:PfafvsDet}

Pfaffian Circuits were introduced as a reformulation of matchcircuits ~\cite{morton2010pfaffian},~\cite{QCtcbSiPT}. We present a slightly different definition using category theory. This is because we want to know what the relation of determinant circuits is with respect to Pfaffian circuits.

We now define the category that gives us Pfaffian circuits. Consider the set $\mathscr{M}\times\{0,1\}$, where $\mathscr{M}$ is the set of labeled skew-symmetric matrices. Furthermore, the columns and rows should have the same labels in the same order. The label sets are subsets of $\mathbb{N}$. As before, for $i \in \N$, let $V_i \isom \mathbb{C}^2$ be spanned by an orthonormal basis (with inner product) $v_{i,0}, v_{i,1}$ and for $\indexset{N} \subset \N$ write $V_\indexset{N} := \ot_{i \in \indexset{N}} V_i$. Now let us consider the following function:

$$\sPf:\mathscr{M}\times\{0,1\}\to V^*_\indexset{N}\otimes V_\indexset{N}$$

$$\sPf(M,0)=\sum_{\indexset{I}\subseteq \indexset{N}}{\Pf(M_{\indexset{I}})|\indexset{I}\rangle}$$

$$\sPf(M,1)=\sum_{\indexset{I}\subseteq \indexset{N}}{\Pf(M_{\bar{\indexset{I}}})\langle \indexset{I}|}$$ 

\noindent where $\ket{\indexset{I}}=\bigotimes_{i \in \indexset{N}} v_{i,\chi(i,\indexset{I})}$, $\bra{\indexset{J}}=\bigotimes_{i \in \indexset{M}} v^*_{i,\chi(i,\indexset{J})}$ and the indicator function $\chi(i,\indexset{I})=0$ if $i \notin \indexset{I}$ and $1$ if $i \in \indexset{I}$. We denote by $M_\indexset{I}$ the principal minor of $M$ with row and column labels $\indexset{I}$. $M_{\bar{\indexset{I}}}$ means the principal minor of $M$ with the rows and columns labeled $\indexset{I}$ removed. We will use the convention that $\sPf(M,0)$ will be denoted $\sPf(M)$ and $\sPf(M,1)$ will be denoted $\sPfd(M)$.

The $\sPf$ function lets us define a monoidal subcategory of $\fdVect_\C$. Let $\msP$ be the free monoidal category defined as follows. The objects are of the form $V_\indexset{N}$ for ordered subsets of $\mathbb{N}$, the tensor product being the usual one. The morphisms of $\msP$ are generated by elements from the image of $\sPf$. Composition and tensor product will be inherited from $\fdVect_\C$.

Suppose we are given a Pfaffian circuit $\Gamma$. Let $\Xi_i$ be the morphisms of the form $\sPf(M)$ and $\Theta_i$ be the morphisms of the form $\sPfd(M)$. We define $\Xi=\tilde{\bigoplus}_i{\Xi_i}$ and $\Theta$ likewise. $\tilde{\bigoplus}$ is the direct sum with the row and columns reordered as follows: The ordering is found by drawing a planar curve through the Pfaffian circuit such that every edge is intersected by the curve once and exactly once. Since a Pfaffian circuit is planar and bipartite, such a curve always exists and the result is independent of the choice of curve. The edges are then labeled based on when the curve intersects them. This is ordering used to define $\tilde{\bigoplus}$. $\widecheck{\Theta}$ is defined to be $\{(-1)^{i+j+1}\theta_{ij}\}$.

\begin{theorem}
 The value of a Pfaffian circuit $\Gamma$ is given by $\Pf(\Xi+\widecheck{\Theta})$~\cite{morton2010pfaffian}
\end{theorem}

Thus Pfaffian circuits can be computed in polynomial time. Now we seek a functor transforming determinantal circuits into Pfaffian circuits. Such a functor should preserve the trace so that the resulting Pfaffian circuit solves the same problem as the original determinantal circuit. The functor should also be faithful.

The morphisms in $\mcD$ from $V_\indexset{n}\to V_\indexset{m}$ are isomorphic to the variety $$D_{n,m}:=\{(1,\dots,\det(M_{\indexset{I}\indexset{J}}),\dots,\det(M))| M\in \textnormal{Mat}_{n\times m}\}$$ given by tuples of minors of $n\times m$ matrices. Then define the variety $$P_{n}:=\{(1,\dots,\Pf(M_{\indexset{I}\indexset{J}}),\dots,\Pf(M)| M\in\mathscr{M}_n\},$$ the tuples of minors of $n\times n$ skew-symmetric matrices. This variety is isomorphic to the image of the $\sPf$ functor on the set $\mathscr{M}_n\times\{0\}$. We first want to find an embedding $D_{n,m}\hookrightarrow P_{n+m}$.

We can assume that $n=m$, otherwise, we pad the matrix with columns or rows of zeros as necessary. So we want to find a map $D_{n,n}\hookrightarrow P_{2n}$. For an $n\times n$ matrix $M$, the following formula is classically known:
\begin{equation*}
\Pf
 \begin{bmatrix}
  0&M\\
-M^T&0
 \end{bmatrix}
=(-1)^{n(n-1)/2}\det(M).
\end{equation*}
This embedding of $M$ into a skew-symmetric matrix is close to the map we are looking for, however this na\"ive way may change the sign on some of the minors of $M$. So we must modify this map slightly. Define $\tilde{M}$ as the matrix $M$ reflected across a vertical axis, and define $S(M)$ to be
\begin{equation*}
 S(M)=
\begin{bmatrix}
  0&\tilde{M}\\
-\tilde{M}^T&0
 \end{bmatrix}.
\end{equation*}
\begin{proposition}
For an $n\times n$ matrix $M$, 
 \begin{equation*}
 \Pf(S(M))=\Pf
\begin{bmatrix}
  0&\tilde{M}\\
-\tilde{M}^T&0
 \end{bmatrix}
= \det(M).
\end{equation*}
\end{proposition}
\begin{proof}
 
 In general, $\tilde{M}$ can be made from $M$ with $\lfloor\frac{n}{2}\rfloor$ column swaps. So if $n\equiv 0,1$ modulo 4, $\lfloor\frac{n}{2}\rfloor$ is an even number and so $\det(\tilde{M})=\det(M)$. Now if $n$ is congruent to 0 or 1 modulo 4, then $\Pf(S(M))=(-1)^{n(n-1)/2}\det(\tilde{M})=\det(\tilde{M})=\det(M)$. If $n$ is congruent to 2 or 3 modulo 4, then $\lfloor\frac{n}{2}\rfloor$ is an odd number so $\det(\tilde{M})=-\det(M)$ and $\Pf(S(M))=(-1)^{n(n-1)/2}\det(\tilde{M})=-\det(\tilde{M})=\det(M)$.

\end{proof}

This map will end up giving us the desired embedding. We also need this map to be a functor. The morphisms of $\mcD$ and $\msP$ look quite different. Note that there are two primary types of morphisms in $\msP$, namely those of the form $\sPf(M)$ and those of form $\sPfd(M)$. Thus Pfaffian circuits form bipartite graphs. Determinantal circuits, on the other hand, are not bipartite at all. There are morphisms from $V_{\indexset{n}}\to V_{\indexset{m}}$ for any sets $\indexset{n}$ and $\indexset{m}$ of any size. 

Given how different these circuits look on the surface, we must really look at the categorical properties of $\msP$ to understand how to construct our functor. The key will be the ability to bend wires in Pfaffian circuits in certain ways. In the language of monoidal categories, we will need our category to have daggers and duals for objects.

\begin{theorem}
 $\msP$ is a strict monoidal category with duals for objects.
\end{theorem}
\begin{proof}
 
By our definition of $\msP$, it will be the smallest monoidal subcategory of $\fdVect_\C$ containing the generating morphisms with the specified objects. A monoidal category $(\mcC,\ot,\lambda,\rho,\alpha)$ is strict if the natural transformations $\lambda$, $\rho$, $\alpha$ are identities. It is a theorem that every monoidal category is equivalent to a strict one ~\cite{maclane1998categories}.

So we can assume without loss of generality that we are working with a strict category equivalent to $\fdVect_\C$ instead. So the $\alpha,\lambda,$ and $\rho$ maps that $\msP$ inherits will be identities. We want to show that the identity morphism is actually generated by our specified morphisms. Consider the following matrix for an object $A$:
\begin{equation*}
I_A=\bordermatrix{ &A&A\cr
A & 0 & 1\cr
A&-1&0\cr}.
\end{equation*} Let $L_A=\sPf(I_A)=|0_A0_A\rangle+|1_A1_A\rangle$ and $R_A=\sPfd(I_A)=\langle0_A0_A|+\langle1_A1_A|$. Then we can contract these these two morphisms along a single edge as in the following picture:

\begin{center}
\begin{tikzpicture}
\draw (-1.25,-.1) node[anchor=east]{$A$};
\draw[->] (-1.25,-.1) -- (0,-.1);
\draw (0,-.1) -- (.5,-.1);
\draw[->] (-.5,-.4) -- (0,-.4);
\draw (0,-.4) -- (.5,-.4);
 \draw (.5,0) rectangle (1.25,-.5) ;
 \draw (.5,-.25) node[anchor=west]{$L_A$};
\draw (-.5,-.3) rectangle (-1.25,-.8);
\draw (-1.25,-.55) node[anchor=west]{$R_A$};
\draw[->] (-.5,-.7) -- (0,-.7); 
\draw (0,-.7)  -- (1.25,-.7);
\draw (1.25,-.7) node[anchor=west]{$A$};
\end{tikzpicture}.
\end{center} This gives us the morphism $|0_A\rangle\langle0_A|+|1_A\rangle\langle1_A|$ which is the identity morphism on $A$. Furthermore, $L_A$ and $R_A$ are the so called "cup" and "cap" morphisms and we have just shown that they satisfy the zig-zag axiom. This shows that $\msP$ has duals for objects.
\end{proof}

\begin{defn}
 The anti-transpose of a matrix $N$, denoted by $\hat{N}$, is $N$ flipped across the non-standard diagonal.
\end{defn}

\begin{lemma}
 $\Pf(\hat{N})=\Pf(N)$.
\end{lemma}
\begin{proof}

Let $N=\{\eta_{ij}\}$ be an $n\times n$ matrix. If $n$ is odd, the above is trivial, so let $n$ be even. Now let $\mathscr{F}$ be the set of partitions of $[n]$ into pairs, $(i_k,j_k)$, $i_k<j_k$. If $\pi\in\mathscr{F}$ we can define the sign of $\pi$, $\tn{sgn}(\pi)$. This is done by considering the set $[n]$ as a sequence of nodes laid out horizontally and labeled $1,\dots,n$ from left to right. Then if two nodes are paired in $\pi$, connect them with an edge. Then $\tn{sgn}(\pi)$ is $(-1)^k$ where $k$ is the number of places where lines cross. Now we can define $\Pf(N)$ as follows:
$$\Pf(N)=\sum_{\pi\in\mathscr{F}}{\tn{sgn}(\pi)\prod_{(i_k,j_k)\in\pi}{\eta_{i_kj_k}}}.$$
Now let $\eta'_{ij}=\eta_{n-j+1,n-i+1}$ be the entries of $\hat{N}$ and suppose $\pi\in\mathscr{F}$. Then the mapping $\mathscr{F}\to\mathscr{F}:\pi\mapsto \pi'$ given by $(i_k,j_k)\mapsto(n-j_k+1,n-i_k+1)$ is a bijective involution. Note that $\pi'$ is the matching formed from $\pi$ by relabeling the nodes as $n,\dots,1$ from left to right. This preserves the number of crossings of edges so that $\tn{sgn}(\pi')=\tn{sgn}(\pi)$. Thus we get
$$\Pf(\hat{N})=\sum_{\pi\in\mathscr{F}}{\tn{sgn}(\pi)\prod_{(i_k,j_k)}{\eta'_{i_kj_k}}}=$$
$$\sum_{\pi'\in\mathscr{F}}{\tn{sgn}(\pi')\prod_{(n-j_k+1,n-i_k+1)}{\eta_{n-j_k+1,n-i_k+1}}}=\Pf(N).$$
 
\end{proof}

\begin{defn}
If $\indexset{I}$ is a bitstring, let $\tilde{\indexset{I}}$ be the bitstring reflected across a vertical axis.. If $\indexset{I}\subseteq \indexset{N}$, $\tilde{\indexset{I}}$ is formed by considering $\indexset{I}$ as a bitstring representing a characteristic function. Then $\tilde{\indexset{I}}$ is a characteristic function defining another subset of $\indexset{N}$.  Then $|\tilde{\indexset{I}}\rangle=\bigotimes_{i \in \indexset{N}} v_{i,\chi(i,\tilde{\indexset{I}})}$ and $\bra{\tilde{\indexset{I}}}=\bigotimes_{i \in \indexset{N}} v^*_{i,\chi(i,\tilde{\indexset{I}})}$
\end{defn}

\begin{corollary}
 Let $N$ be a skew symmetric matrix with labels $\indexset{M}$. Let $\hat{N}$ also have labels $\indexset{M}$. $\sPf(\hat{N})=\sum_{\indexset{I}\subseteq \indexset{M}}{\Pf(N_\indexset{I})|\tilde{\indexset{I}}\rangle}$
\end{corollary}
\begin{proof}
 
Let $\indexset{I}\subseteq \indexset{M}$. Note that $N_\indexset{I}=\hat{N}_{\tilde{\indexset{I}}}$. Then $\Pf(N_\indexset{I})=\Pf(\hat{N}_{\tilde{\indexset{I}}})$. This gives the result.

\end{proof}

\begin{example}
Consider the following matrix:
 \begin{equation*}
 N=
  \begin{pmatrix}
   0&0&a&0\\
   0&0&0&b\\
   -a&0&0&0\\
   0&-b&0&0
  \end{pmatrix}
 \end{equation*}
$$\sPf(\hat{N})=|0000\rangle+b|1010\rangle+a|0101\rangle-ab|1111\rangle$$
$$=\Pf(N_{\emptyset})|0000\rangle+\Pf(N_{\{2,4\}})|1010\rangle+\Pf(N_{\{1,3\}})|0101\rangle+\Pf(N)|1111\rangle$$
$$=\sum_{\indexset{I}\subseteq \indexset{M}}{\Pf(N_\indexset{I})|\tilde{\indexset{I}}\rangle}.$$
\end{example}

\begin{proposition}
 For any skew-symmetric matrix $M$, 
\[\sum_{\indexset{I}}{\Pf(M_\indexset{I})\langle \indexset{I}|}\]\[
\sum_{\indexset{I}}{\Pf(M_{\bar{\indexset{I}}})|\indexset{I}\rangle}\] are morphisms of $\msP$. This implies that $\msP$ is a dagger monoidal category.
\end{proposition}
\begin{proof}
 
Let $M$ have labels $A=\{A_1,\dots,A_n\}$. Then $\hat{M}$ will have labels $\hat{A}=\{A_n,\dots,A_1\}$. Let $R_{A}$ be defined as:
\begin{equation*}
 R_A=\sPf\bordermatrix{ &\hat{A}&A\cr
\hat{A} & 0 & \tilde{I}\cr
A & -\tilde{I}^T&0}
\end{equation*}
where $\tilde{I}$ is the identity matrix reflected over a vertical axis. Then consider the following morphism in $\msP$:

\begin{center}
\begin{tikzpicture}

\draw (.5,-1.23) rectangle (1.9,-2.2);
\draw (1.2,-1.7) node{$\sPf(\hat{M})$};

\draw[->] (-1.25,-1.5) -- (0,-1.5);
\draw (0,-1.5) -- (.5,-1.5);
\draw (-.3,-1.2) node{$A_n$};

\draw (-.18,-1.71) node[anchor=east]{$\vdots$};
\draw[->] (-1.25,-2.1) -- (0,-2.1);
\draw (0,-2.1) -- (.5,-2.1);
\draw (-.3,-2.4) node{$A_1$};

\draw (-2.1,-1.23) rectangle (-1.25,-3.5);
\draw (-1.25,-2.25) node[anchor=east]{$R_{A}$};	

\draw[->] (-1.25,-3.4) -- (0,-3.4);
\draw (0,-3.4) -- (1.25,-3.4);
\draw[->] (-1.25,-2.75) -- (0,-2.75);
\draw (0,-2.75) -- (1.25,-2.75);
\draw (0,-3) node[anchor=east]{$\vdots$};
\draw (1.25,-3.4) node[anchor=west]{$A_n$};
\draw (1.25,-2.75) node[anchor=west]{$A_1$};
\end{tikzpicture}.
\end{center}
This diagram represents the morphism
\[\bigg(\sum_{\indexset{I}\subseteq A}{\Pf(M_\indexset{I})|\tilde{\indexset{I}}\rangle}\bigg)\bigg(\sum_{\indexset{I}\subseteq\{\hat{A}A\}}{\langle\tilde{\indexset{I}}|\langle \indexset{I}|}\bigg)=\]\[
\sum_{\indexset{I}\subseteq A}{\Pf(M_\indexset{I})\langle \indexset{I}|}.\]

We can similarly form $\sum_{\indexset{I}}{\Pf(M_{\bar{\indexset{I}}})|\indexset{I}\rangle}$ by instead using $\sPfd(\hat{M})$ and $\sPf(R_A)$. Now since every generating morphism has a dagger, the entire category has a dagger and it is the usual vector space dagger.
\end{proof}

\begin{theorem} \label{thm:detarepfaf}
 Every morphism in $\mcD$ is a morphism in $\msP$. Thus there is a trace-preserving faithful strict monoidal functor from $\mcD\to\msP$ given by inclusion.
\end{theorem}
 \begin{proof}
First suppose that $M$ is an $n\times n$ matrix. The labels of $S(M)=R\cup\tilde{C}$ where $R$ is the row labels of $M$ and $C$ are the column labels of $M$. Now let $K$ be a subset of the labels. Then let $I=K\cap R$ and $\tilde{J}=K\cap\tilde{C}$. Then we get
\begin{equation*}
\Pf(S(M)_K)=\Pf
 \begin{bmatrix}
  0& \tilde{M}_{I,J}\\
-\tilde{M}_{I,J}^T&0
 \end{bmatrix}
=\det(M_{I,J}),
\end{equation*}
so that 
$$\sPf(S(M))=\sum_{I\subseteq R,J\subseteq C}{\det(M_{I,J}|I\rangle|\tilde{J}\rangle}$$
$$\sPfd(S(M))=\sum_{I\subseteq R,J\subseteq C}{\det(M_{\bar{I},\bar{J}})\langle I|\langle\tilde{J}|}.$$

The identity morphism on $A_n\otimes\cdots\otimes A_1$ in $\mcC$ is given by the matrix
\begin{equation*}
I_{\otimes{A_i}}=
 \bordermatrix{ & A_n &A_{n-1}&\cdots&A_1\cr
A_n & 1 & 0&\cdots&0\cr
A_{n-1} &0&1&\cdots&0\cr
\vdots&\vdots&\vdots&\cdots&\vdots\cr
A_1 &0&0&\cdots&1}.
\end{equation*}

Suppose we have an $n\times n$ matrix $M:B_1\otimes\cdots\ot B_n\to A_1\ot\cdots\otimes A_n$. Then we define $M_*=\sPf(S(M))$ and $R_{\otimes{A_i}}=\sPfd(S(I_{\otimes{A_i}}))$. Let us consider the morphism in $\msP$ given by

\begin{center}
\begin{tikzpicture}
\draw (-2.2,-.1) node[anchor=east]{$B_1$};
\draw[->] (-2.2,-.1) -- (0,-.1);
\draw (0,-.1) -- (.5,-.1);
\draw (0,-.3) node[anchor=east]{$\vdots$};
\draw (-2.2,-.7) node[anchor=east]{$B_n$};
\draw[->] (-2.2,-.7) -- (0,-.7);
\draw (0,-.7) -- (.5,-.7);
\draw (.5,0) rectangle (1.25,-2.2);
\draw (.9,-1.1) node{$M_*$};

\draw[->] (-1.25,-1.5) -- (0,-1.5);
\draw (0,-1.5) -- (.5,-1.5);
\draw (-.3,-1.2) node{$A_n$};

\draw (-.18,-1.71) node[anchor=east]{$\vdots$};
\draw[->] (-1.25,-2.1) -- (0,-2.1);
\draw (0,-2.1) -- (.5,-2.1);
\draw (-.3,-2.4) node{$A_1$};

\draw (-2.2,-1.23) rectangle (-1.25,-3.5);
\draw (-1.21,-2.25) node[anchor=east]{$R_{\otimes{A_i}}$};	

\draw[->] (-1.25,-3.4) -- (0,-3.4);
\draw (0,-3.4) -- (1.25,-3.4);
\draw[->] (-1.25,-2.75) -- (0,-2.75);
\draw (0,-2.75) -- (1.25,-2.75);
\draw (0,-3) node[anchor=east]{$\vdots$};
\draw (1.25,-3.4) node[anchor=west]{$A_n$};
\draw (1.25,-2.75) node[anchor=west]{$A_1$};
\end{tikzpicture}.
\end{center}
For $I\subseteq\{B_1,\dots, B_n\}$, $\tilde{J},\tilde{J}'\subseteq \{A_n,\dots, A_1\}$; and $J'\subseteq\{A_1,\dots A_n\}$, we can represent this tensor as 
$$\bigg(\sum{\det(M_{I,J})|I\rangle|\tilde{J}\rangle}\bigg)\bigg(\sum{\langle\tilde{J}'|\langle J'|}\bigg)=$$
$$\sum{\det(M_{I,J})|I\rangle\langle J|}=\sDet(M).$$
 So for any square matrix $M$, $\sDet(M)$ is a morphism in $\msP$. Now not every morphism in $\mcC$ is a square matrix. However, if we have an $n\times m$ matrix $M$, we can make it square. If $n<m$, then let $M'=M\oplus Z_{m-n}$ where $Z_{m-n}$ is the $(m-n)\times 0$ matrix. If $m<n$, then let $M'=M\oplus Z'_{n-m}$ where $Z'_{n-m}$ is the $0\times (n-m)$ matrix. What this amounts to is either adding rows or columns of zeros as needed.

Now note that $\sPf([0])=|0\rangle$. $\langle0|$ is also a morphism in $\msP$. Consider $\sPfd(K)=\langle0_A0_B|+\langle1_A1_B|$ where
\begin{equation*}
 K=\bordermatrix{ & A & B\cr
A & 0 & 1\cr
B & -1 & 0},
\end{equation*}
and contracting this with the morphism $|0_B\rangle$, we obtain $\langle0_A|$.

 Let $M$ be an arbitrary $n\times m$ matrix. Then let us consider $S(M')$ where $M'$ is defined as above. Suppose $n<m$. Then
$$\sPf(S(M'))=\sum_{I,J}{\det(M_{I,J})|I0_{n+1}\cdots0_{m}\rangle\langle J|}$$

\noindent Consider the following diagram in $\msP$:

\begin{center}
\begin{tikzpicture}
 \draw[->] (-1,0) -- (-.5,0);
\draw (-.5,0) -- (0,0);
\draw (-1,0) node[anchor=east]{1};
\draw (-.5,-.3) node{$\vdots$};
\draw[->] (-1,-.7) -- (-.5,-.7);
\draw (-.5,-.7) -- (0,-.7);
\draw (-1,-.7) node[anchor=east]{$n$};
\draw [->] (-1,-1) -- (-.5,-1);
\draw (-.5,-1) -- (0,-1);
\draw (-1,-1) node[anchor=east]{$\langle0|$};
\draw (-.5,-1.3) node{$\vdots$};
\draw[->] (-1,-1.7) -- (-.5,-1.7);
\draw (-.5,-1.7) -- (0,-1.7);
\draw (-1,-1.7) node[anchor=east]{$\langle0|$};
\draw (0,.1) rectangle (2,-1.9);
\draw (1,-.8) node{$\sPf(S(M'))$};
\draw[->] (2,0) -- (2.5,0);
\draw (2.5,0) -- (3,0);
\draw (3,0) node[anchor=west]{1};
\draw (2.5,-.8) node{$\vdots$};
\draw[->] (2,-1.7) -- (2.5,-1.7);
\draw (2.5,-1.7) -- (3,-1.7);
\draw (3,-1.7) node[anchor=west]{$m$};
\end{tikzpicture}.
\end{center}
The morphism this represents will obviously come out to be $\sDet(M)$. If $n>m$, then copies of $|0\rangle$ are added to the extra output wires of $\sPf(S(M'))$. Thus we have finished the proof of theorem. Every morphism of $\mcD$ is in fact a morphism in $\msP$. Furthermore, the reinterpretation of a determinantal circuit as a Pfaffian circuit can obviously be done in polynomial time. 
\end{proof}

Despite this fact, determinantal circuits still have some advantages. If a Pfaffian circuit can be represented as a determinantal circuit, its evaluation will be more efficient. Suppose we have a determinantal circuit with a single morphism from $V_{\indexset{N}}\to V_{\indexset{N}}$ with $|N|=n$. Then we are computing a determinant of an $n\times n$ matrix $M$. Embedding this into a Pfaffian circuit would look like the following (with a bit of simplification):
\begin{center}
\begin{tikzpicture}

\draw[->] (-1.25,-.25) -- (-.25,-.25);
\draw (-.25,-.25) -- (.5,-.25);
\draw (-.1,-1.1) node[anchor=east]{$\vdots$};

\draw (.5,0) rectangle (1.25,-2.2);
\draw (.9,-1.1) node{$M_*$};

\draw (-.3,.1) node{$A_n$};

\draw[->] (-1.25,-2) -- (-.25,-2);
\draw (-.25,-2) -- (.5,-2);
\draw (-.3,-2.4) node{$A_1$};

\draw (-2.2,0) rectangle (-1.25,-2.2);
\draw (-1.21,-1.1) node[anchor=east]{$R_{\otimes{A_i}}$};	

\end{tikzpicture}.
\end{center}
where $M_*=\sPf(S(M))$ as before. Since the evaluation of this circuit involves computing the Pfaffian of a $2n\times 2n$ skew-symmetric matrix, we see a quadratic increase in the time complexity. Consider the case of computing Tutte polynomials of lattice path matriods with Pfaffian circuits. The algorithm will have time complexity $O(n^8)$ using Pfaffian circuits, which is not an improvement over the algorithm given in ~\cite{bonin2007multi}.

\section*{Acknowledgments}
J.M. and J.T. were supported in part by the Defense Advanced Research Projects Agency under Award No. N66001-10-1-4040. Portions of J.T.'s work were sponsored by the Applied Research Laboratory's Exploratory and Foundational Research Program.

\bibliographystyle{plain}
\bibliography{bibfile} 

\def\cdprime{$''$} \def\Dbar{\leavevmode\lower.6ex\hbox to 0pt{\hskip-.23ex
  \accent"16\hss}D} \def\cprime{$'$} \def\cprime{$'$} \def\cprime{$'$}
  \def\cprime{$'$} \def\Dbar{\leavevmode\lower.6ex\hbox to 0pt{\hskip-.23ex
  \accent"16\hss}D} \def\cprime{$'$}
\begin{thebibliography}{10}

\bibitem{arora2009computational}
S.~Arora and B.~Barak.
\newblock {\em Computational complexity: a modern approach}.
\newblock Cambridge University Press, 2009.

\bibitem{bergholm2011categorical}
V.~Bergholm and J.D. Biamonte.
\newblock Categorical quantum circuits.
\newblock {\em Journal of Physics A: Math. Theor.}, 44:245304, 2011.

\bibitem{bonin2003lattice}
J.~Bonin, A.~De~Mier, and M.~Noy.
\newblock {Lattice path matroids: enumerative aspects and Tutte polynomials}.
\newblock {\em Journal of Combinatorial Theory, Series A}, 104(1):63--94, 2003.

\bibitem{bonin2007multi}
J.E. Bonin and O.~Gimenez.
\newblock {Multi-path matroids}.
\newblock {\em Combinatorics, Probability and Computing}, 16(02):193--217,
  2007.

\bibitem{bruns2011relations}
W.~Bruns, A.~Conca, and M.~Varbaro.
\newblock Relations between the minors of a generic matrix.
\newblock {\em Arxiv preprint arXiv:1111.7263}, 2011.

\bibitem{bulatov2010complexity}
A.~Bulatov.
\newblock {The complexity of the counting constraint satisfaction problem}.
\newblock {\em Automata, Languages and Programming}, pages 646--661, 2010.

\bibitem{cai2012complexity}
J.Y. Cai and X.~Chen.
\newblock Complexity of counting {CSP} with complex weights.
\newblock In {\em Proceedings of the 44th symposium on Theory of Computing},
  pages 909--920. ACM, 2012.

\bibitem{MR1401006}
F.~R.~K. Chung and Robert~P. Langlands.
\newblock A combinatorial {L}aplacian with vertex weights.
\newblock {\em J. Combin. Theory Ser. A}, 75(2):316--327, 1996.

\bibitem{Damm03}
C.~Damm, M.~Holzer, and P.~McKenzie.
\newblock The complexity of tensor calculus.
\newblock {\em Computational Complexity}, 11(1):54–--89, 2003.

\bibitem{horn1990matrix}
R.A. Horn and C.R. Johnson.
\newblock {\em Matrix analysis}.
\newblock Cambridge University Press, 1990.

\bibitem{joyal1991geometry}
A.~Joyal and R.~Street.
\newblock {The geometry of tensor calculus. I}.
\newblock {\em Advances in Mathematics}, 88(1):55--112, 1991.

\bibitem{joyal1996traced}
A.~Joyal, R.~Street, and D.~Verity.
\newblock Traced monoidal categories.
\newblock In {\em Mathematical Proceedings of the Cambridge Philosophical
  Society}, volume 119, pages 447--468. Cambridge University Press, 1996.

\bibitem{kassel1995quantum}
C.~Kassel.
\newblock {\em Quantum groups}, volume 155.
\newblock Springer, 1995.

\bibitem{lafont2003towards}
Y.~Lafont.
\newblock Towards an algebraic theory of {B}oolean circuits.
\newblock {\em Journal of Pure and Applied Algebra}, 184(2):257--310, 2003.

\bibitem{landsberg2012holographic}
JM~Landsberg, J.~Morton, and S.~Norine.
\newblock Holographic algorithms without matchgates.
\newblock {\em Linear Algebra and its Applications}, 438(15), 2013.

\bibitem{lloyd2011closed}
S.~Lloyd, L.~Maccone, R.~Garcia-Patron, V.~Giovannetti, Y.~Shikano,
  S.~Pirandola, L.A. Rozema, A.~Darabi, Y.~Soudagar, L.K. Shalm, and
  A.~Steinberg.
\newblock Closed timelike curves via postselection: theory and experimental
  test of consistency.
\newblock {\em Physical Review Letters}, 106(4):40403, 2011.

\bibitem{maclane1998categories}
S.~Mac~Lane.
\newblock {\em {Categories for the working mathematician}}.
\newblock Springer verlag, 1998.

\bibitem{morton2010pfaffian}
J.~Morton.
\newblock Pfaffian circuits.
\newblock {\em Arxiv preprint arXiv:1101.0129}, 2010.

\bibitem{morton2012undecidability}
J.~Morton and J.~Biamonte.
\newblock Undecidability in tensor network states.
\newblock {\em Physical Review A}, 86(3):030301, 2012.

\bibitem{morton2014belief}
Jason Morton.
\newblock Belief propagation in monoidal categories.
\newblock {\em Quantum Physics and Logic 2014, arXiv:1405.2618}, 2014.

\bibitem{morton2013computing}
Jason Morton and Jacob Turner.
\newblock Computing the tutte polynomial of lattice path matroids using
  determinantal circuits.
\newblock {\em arXiv preprint arXiv:1312.3537}, 2013.

\bibitem{mulmuley2001geometric}
K.D. Mulmuley and M.~Sohoni.
\newblock Geometric complexity theory i: An approach to the {P} vs. {NP} and
  related problems.
\newblock {\em SIAM Journal on Computing}, 31(2):496--526, 2001.

\bibitem{selinger2009survey}
P.~Selinger.
\newblock {A survey of graphical languages for monoidal categories}.
\newblock {\em New Structures for Physics}, pages 275--337, 2009.

\bibitem{QCtcbSiPT}
L.~Valiant.
\newblock Quantum circuits that can be simulated classically in polynomial
  time.
\newblock {\em SIAM J. Comput.}, 31(4):1229--1254, 2002.

\end{thebibliography}
\end{document}